\documentclass{cstr}

\usepackage{amsmath,amsfonts,amssymb,latexsym,bm}
\usepackage{amsthm}
\usepackage{mathrsfs}
\usepackage{graphicx}
\usepackage{subfig}
\usepackage{url}
\usepackage{cite}

\newtheorem{proposition}{Proposition}
\newtheorem{conjecture}{Conjecture}

\usepackage{algorithm, algorithmic}
\renewcommand{\algorithmicrequire}{{\bf Input:}}
\renewcommand{\algorithmicensure}{{\bf Output:}}

\title{
Efficient implementations of the modified Gram-Schmidt orthogonalization with a non-standard inner product
}

\author[1]{Akira Imakura}
\author[2]{Yusaku Yamamoto}

\affil[1]{University of Tsukuba, Japan}
\affil[2]{The University of Electro-Communications, Japan}

\email{imakura@cs.tsukuba.ac.jp}

\begin{document}
\maketitle
\thispagestyle{titlepage}

\begin{abstract}
The modified Gram-Schmidt (MGS) orthogonalization is one of the most well-used algorithms for computing the thin QR factorization.
MGS can be straightforwardly extended to a non-standard inner product with respect to a symmetric positive definite matrix $A$.
For the thin QR factorization of an $m \times n$ matrix with the non-standard inner product, a naive implementation of MGS requires $2n$ matrix-vector multiplications (MV) with respect to $A$.
In this paper, we propose $n$-MV implementations: a high accuracy (HA) type and a high performance (HP) type, of MGS.
We also provide error bounds of the HA-type implementation.
Numerical experiments and analysis indicate that the proposed implementations have competitive advantages over the naive implementation in terms of both computational cost and accuracy.
\end{abstract}

\section{Introduction}
In this paper, we consider computing the thin QR factorization with a non-standard inner product of the form
\begin{equation}
        Z = QR, \quad
        Q^{\rm T} A Q = I_n,
        \label{eq:oqr}
\end{equation}
where $Z, Q \in \mathbb{R}^{m \times n}$ $(m \geq n), R \in \mathbb{R}^{n \times n}$ and $A \in \mathbb{R}^{m \times m}$ is symmetric positive definite (spd).
This type of QR factorization with a non-standard inner product \eqref{eq:oqr} appears in weighted least squares problems \cite{Bjorck:1996, Gulliksson:1995}, projection methods for solving symmetric generalized eigenvalue problems \cite{Knyazev:2001,Imakura:2017}, the weighted (block) GMRES and FOM methods \cite{Essai:1998,Imakura:2013} and so on.
\par
For the standard inner product, i.e., $A = I_m$, there are several established algorithms for computing the thin QR factorization \cite{Trefethen:1997, Bjorck:1996}.
These methods can be classified into two groups: orthogonal triangularization methods such as the Householder transformation and triangular orthogonalization methods such as the Gram-Schmidt orthogonalization and the Cholesky QR algorithm.
An extension of the Householder transformation for a quasimatrix has been developed by Trefethen \cite{Trefethen:2009} and it was shown to be applicable to \eqref{eq:oqr} \cite{Yanagisawa:2014}.
However, Trefethen's Householder-type QR algorithm for \eqref{eq:oqr} requires some $A$-orthonormal basis that is a big issue to use it for general $A$.
In contrast, the methods in the second group can be straightforwardly extended to a non-standard inner product.
The error bounds of these methods are also well analyzed in \cite{Rozloznik:2012, Lowery:2014, Yamamoto:2016}.
\par
Here, we focus on the modified Gram-Schmidt (MGS) orthogonalization.
For a standard inner product, the number of floating-point operations (flops) of MGS is $2mn^2$.
For the non-standard inner product, naive implementations of MGS additionally require $2n$ matrix-vector multiplications (MV) with respect to $A$ \cite{Stewart:2001}, which is the most-time consuming part for general $A$.
\par
In this paper, we aim to reduce the computational cost of MGS.
We propose high accuracy (HA) type and high performance (HP) type implementations of MGS that require only $n$ MV.
We also provide error bounds of the HA-type implementation.
One can also apply the proposed concept to the classical Gram-Schmidt (CGS) orthogonalization for its $n$-MV implementations.
\par
The remainder of this paper is organized as follows.
In Section~\ref{sec:propose}, we estimate the minimal computational cost for MGS and propose efficient implementations of MGS.
We present error bounds of the proposed implementation in Section~\ref{sec:error}.
Numerical results are reported in Section~\ref{sec:experiment}.
Section~\ref{sec:conclusion} concludes the paper.
\par
Throughout, the following notations are used.
Let $A \in \mathbb{R}^{m \times m}$ be spd and ${\bm x}, {\bm y} \in \mathbb{R}^m$.
Then, the $A$-inner product of vectors ${\bm x}$ and ${\bm y}$ is defined as $({\bm x}, {\bm y})_A := {\bm x}^{\rm T} A {\bm y}$.
Also, $\| {\bm x} \|_A := \sqrt{({\bm x}, {\bm x})_A} = \sqrt{{\bm x}^{\rm T}A{\bm x}}$ is the corresponding $A$-norm.
Norms without a subscript denote the 2-norm: $\|{\bm x}\| := \|{\bm x}\|_2$ and $\|A\| := \|A\|_2$.
Frobenius norm of a matrix $A$ is denoted by~$\|A\|_{\rm F}$.
For $Z = [{\bm z}_1, {\bm z}_2, \ldots, {\bm z}_n] \in \mathbb{R}^{m \times n}$, we define the range space of the matrix $Z$ by $\mathcal{R}(Z) := {\rm span}\{ {\bm z}_1, {\bm z}_2,$ $ \dots, {\bm z}_n \}$.
If $Z$ is of full column rank, then $\kappa(Z) := \sigma_1 / \sigma_n$ is the condition number of $Z$, where $\sigma_1, \sigma_n$ are the largest and smallest non-zero singular values of~$Z$.
\section{Efficient implementations of MGS}
\label{sec:propose}
\begin{algorithm}[t]
\caption{MGS(col): The column-oriented MGS}
\label{alg:mgs-col}
\begin{minipage}{0.5\hsize}
\begin{algorithmic}[1]
        \FOR{$j = 1, 2, \dots, n$}
        \FOR{$i = 1, 2, \dots, j-1$}
        \STATE $r_{ij} = ({\bm q}_i, {\bm z}_j^{(i-1)})_A$
        \STATE ${\bm z}_j^{(i)} = {\bm z}_j^{(i-1)} - r_{ij} {\bm q}_i$
        \ENDFOR
        \STATE $r_{jj} = \|{\bm z}_j^{(j-1)}\|_A$
        \STATE ${\bm q}_j = {\bm z}_j^{(j-1)} / r_{jj}$
        \ENDFOR  
\end{algorithmic}
\end{minipage}
\begin{minipage}{0.5\hsize}
\centering
\includegraphics[bb = 0 28 308 342, scale=0.3]{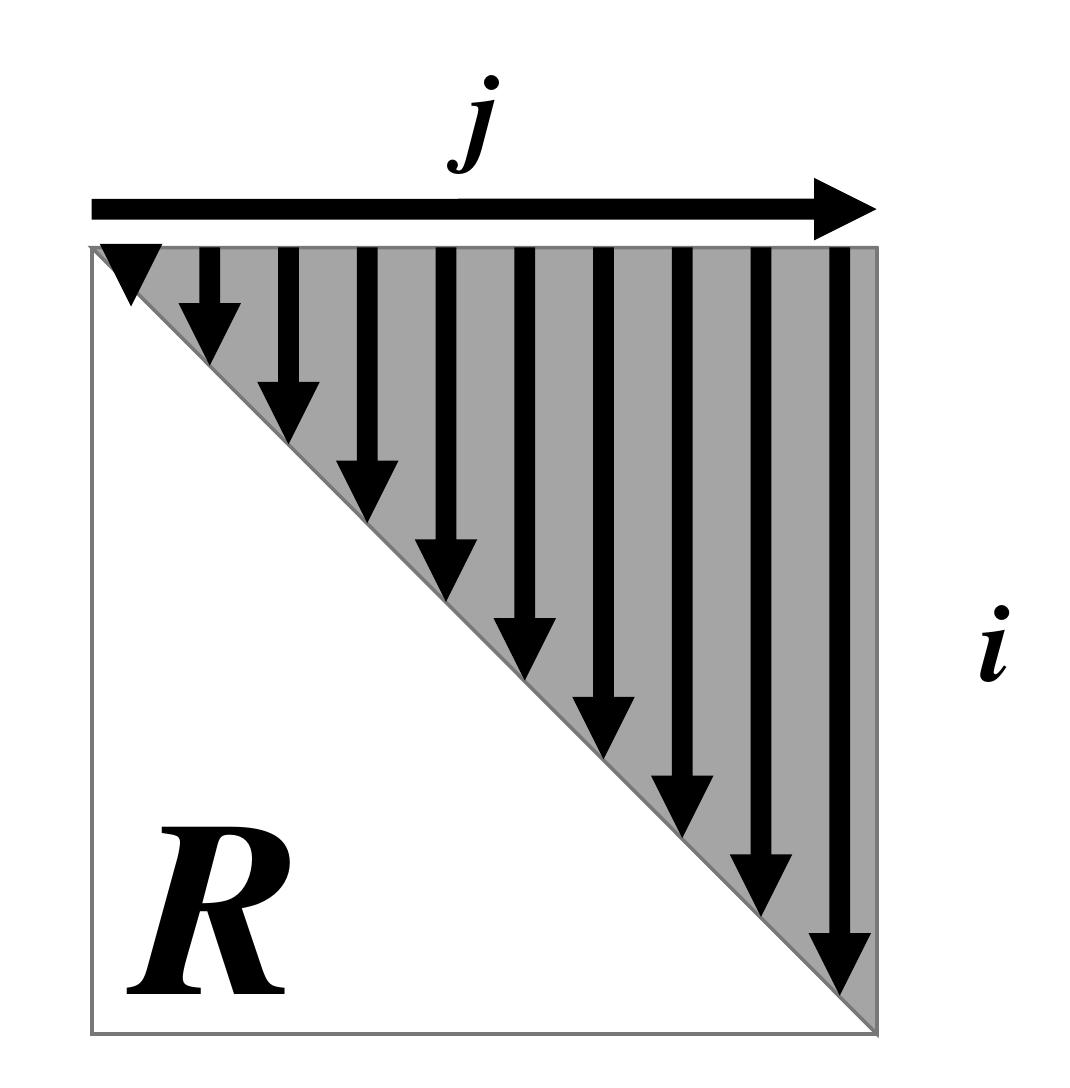} \\
Computing order of $r_{ij}$.
\end{minipage}
\end{algorithm}
\begin{algorithm}[t]
\caption{MGS(row): The row-oriented MGS}
\label{alg:mgs-row}
\begin{minipage}{0.5\hsize}
\begin{algorithmic}[1]
        \FOR{$i = 1, 2, \dots, n$}
        \STATE $r_{ii} = \| {\bm z}_i^{(i-1)} \|_A$
        \STATE ${\bm q}_i = {\bm z}_i^{(i-1)} / r_{ii}$
        \FOR{$j = i+1, i+2, \dots, n$}
        \STATE $r_{ij} = ({\bm q}_i, {\bm z}_j^{(i-1)})_A$
        \STATE ${\bm z}_j^{(i)} = {\bm z}_j^{(i-1)} - r_{ij} {\bm q}_i$
        \ENDFOR
        \ENDFOR  
\end{algorithmic}
\end{minipage}
\begin{minipage}{0.5\hsize}
\centering
\includegraphics[bb = 0 28 308 342, scale=0.3]{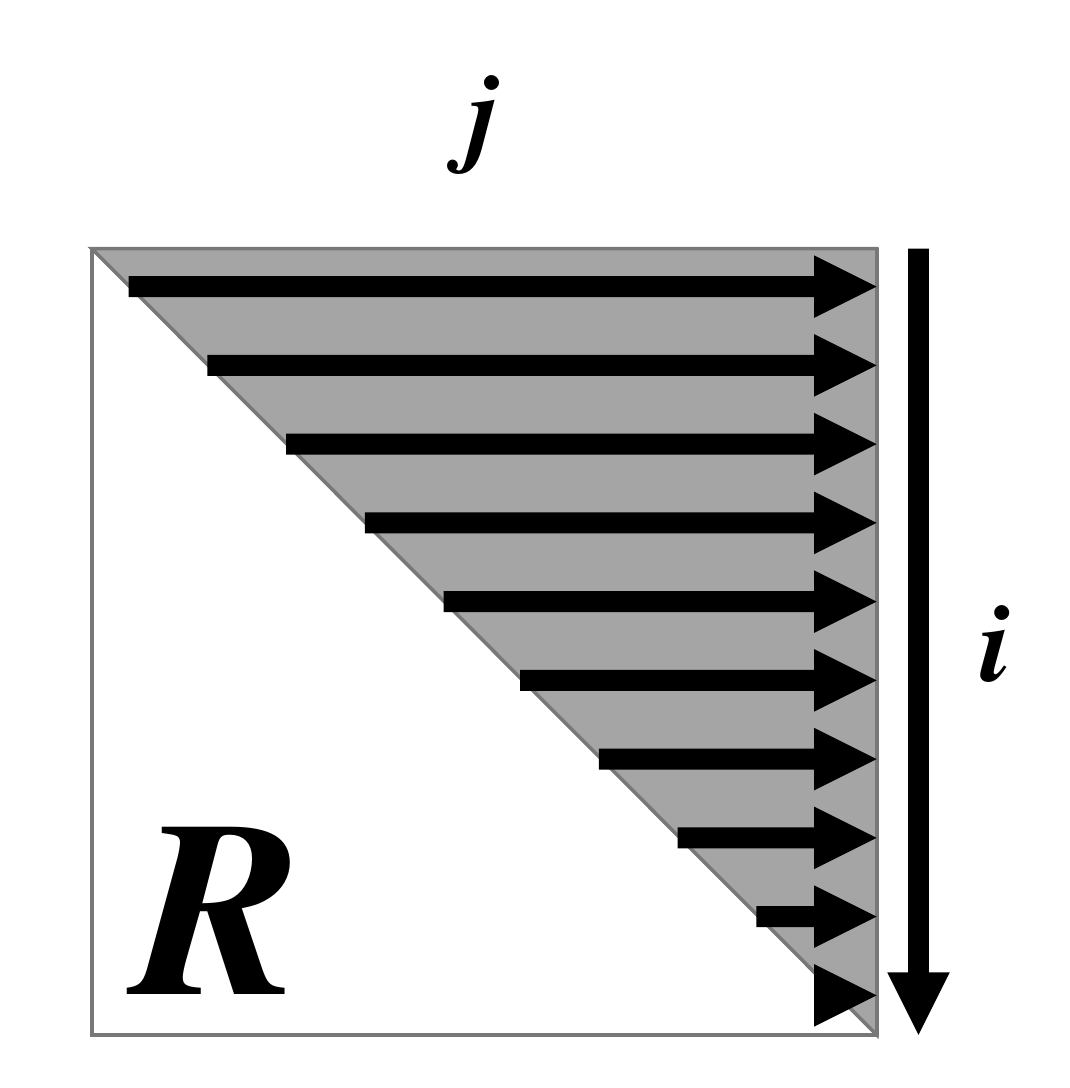} \\
Computing order of $r_{ij}$.
\end{minipage}
\end{algorithm}
There are two types of implementations of MGS: a column-oriented (left-looking) version and a row-oriented (right-looking) version; see Algorithms~\ref{alg:mgs-col} and \ref{alg:mgs-row} \cite{Bjorck:1996}.
In this section, we firstly introduce naive implementations with $2n$ MV.
Then, we estimate the minimal computational cost for MGS and propose efficient implementations of MGS.
\subsection{Naive implementations with $2n$ MV}
\begin{algorithm}[t]
\caption{MGS-naive(col): The naive implementation of the column-oriented MGS}
\label{alg:mgs-naive-col}
\begin{algorithmic}[1]
        \REQUIRE $Z = [{\bm z}_1^{(0)}, {\bm z}_2^{(0)}, \dots, {\bm z}_n^{(0)}] \in \mathbb{R}^{m \times n}, A \in \mathbb{R}^{m \times m}$, where $A$ is spd.
        \ENSURE  $Q = [{\bm q}_1, {\bm q}_2, \dots, {\bm q}_n] \in \mathbb{R}^{m \times n}, R = \{ r_{ij} \}_{1 \leq i,j \leq n} \in \mathbb{R}^{n \times n}$
        \FOR{$j = 1, 2, \dots, n$}
        \FOR{$i = 1, 2, \dots, j-1$}
        \STATE $r_{ij} = {\bm p}_i^{\rm T} {\bm z}_j^{(i-1)}$
        \STATE ${\bm z}_j^{(i)} = {\bm z}_j^{(i-1)} - r_{ij} {\bm q}_i$
        \ENDFOR
        \STATE \fbox{ ${\bm x}_j^{(j-1)} = A {\bm z}_j^{(j-1)}$ }
        \STATE $r_{jj} = \sqrt{ ({\bm z}_j^{(j-1)})^{\rm T} {\bm x}_j^{(j-1)} }$
        \STATE ${\bm q}_j = {\bm z}_j^{(j-1)} / r_{jj}$
        \STATE \fbox{ ${\bm p}_j = A {\bm q}_j$ }
        \ENDFOR  
\end{algorithmic}
\end{algorithm}
\begin{algorithm}[t]
\caption{MGS-naive(row): The naive implementation of the row-oriented MGS}
\label{alg:mgs-naive-row}
\begin{algorithmic}[1]
        \REQUIRE $Z = [{\bm z}_1^{(0)}, {\bm z}_2^{(0)}, \dots, {\bm z}_n^{(0)}] \in \mathbb{R}^{m \times n}, A \in \mathbb{R}^{m \times m}$, where $A$ is spd.
        \ENSURE  $Q = [{\bm q}_1, {\bm q}_2, \dots, {\bm q}_n] \in \mathbb{R}^{m \times n}, R = \{ r_{ij} \}_{1 \leq i,j \leq n} \in \mathbb{R}^{n \times n}$
        \FOR{$i = 1, 2, \dots, n$}
        \STATE \fbox{ ${\bm x}_i^{(i-1)} = A {\bm z}_i^{(i-1)}$ }
        \STATE $r_{ii} = \sqrt{ ({\bm z}_i^{(i-1)})^{\rm T} {\bm x}_i^{(i-1)} }$
        \STATE ${\bm q}_i = {\bm z}_i^{(i-1)} / r_{ii}$
        \STATE \fbox{ ${\bm p}_i = A {\bm q}_i$ }
        \FOR{$j = i+1, i+2, \dots, n$}
        \STATE $r_{ij} = {\bm p}_i^{\rm T} {\bm z}_j^{(i-1)}$
        \STATE ${\bm z}_j^{(i)} = {\bm z}_j^{(i-1)} - r_{ij} {\bm q}_i$
        \ENDFOR
        \ENDFOR  
\end{algorithmic}
\end{algorithm}
\par
For the standard inner product, there is no numerical difference between the column- and row-oriented versions regarding computational cost, memory requirements and accuracy.
Because the operations and rounding errors are the same, they produce exactly the same numerical results.
On the other hand, each one has different advantages for using.
The column-oriented MGS has advantages for successive orthogonalization and reorthogonalization; in contrast, the row-oriented MGS is suitable for column pivoting.
\par
However, the situation is different for a non-standard inner product regarding computational cost and memory requirements.
In naive implementations that uses no auxiliary vectors, the row-oriented MGS requires $2n$ MV; in contrast, the column-oriented MGS requires $\mathcal{O}(n^2)$ MV to compute the $A$-inner products:
\begin{equation*}
        r_{ij} = ({\bm q}_i, {\bm z}_j^{(i-1)})_A = {\bm q}_i^{\rm T} (A {\bm z}_j^{(i-1)}) \quad (i < j),
\end{equation*}
because ${\bm z}_j^{(i-1)}$ depends on both $i$ and $j$ \cite{Stewart:2001, Zhao:2013}.
On the other hand, if storing $n$ auxiliary vectors $A{\bm q}_j, j = 1, 2, \dots, n$, is allowed, then the number of MV of the column-oriented MGS is reduced to $2n$ by computing $r_{ij}$ as
\begin{equation*}
        r_{ij} = ({\bm q}_i, {\bm z}_j^{(i-1)})_A = (A{\bm q}_i)^{\rm T} {\bm z}_j^{(i-1)} \quad (i < j),
\end{equation*}
because ${\bm q}_i$ depends only on $i$ \cite{Stewart:2001}.
This achieves a $2n$-MV implementation of the column-oriented MGS.
\par
Because the computational cost of $\mathcal{O}(n^2)$ MV is unreasonably large, we generally use the $2n$-MV implementations.
Naive implementations with $2n$ MV of the column- and row-oriented MGS are shown in Algorithms~\ref{alg:mgs-naive-col} and \ref{alg:mgs-naive-row}, respectively.
Here, we note that they have the same computational cost and produce exactly the same numerical results.
%
%
\subsection{Estimation of the minimal computational costs}
In MGS, MV with respect to $A$ is used only for computing the $A$-inner products and $A$-norms to construct the elements of $R$, 
\begin{align*}
        &r_{ij} = ({\bm q}_i, {\bm z}_j^{(i-1)})_A \quad
        (i < j), \\
        &r_{jj} = \| {\bm z}_j^{(j-1)} \|_A = \sqrt{({\bm z}_j^{(j-1)}, {\bm z}_j^{(j-1)})_A}.
\end{align*}
Then, we have the following proposition.
\begin{proposition}
        \label{prop:min}
        For each element $r_{ij}$ of $R$ in \eqref{eq:oqr}, there exist ${\bm a} \in \mathcal{R}(Z), {\bm b} \in \mathcal{R}(AZ)$ such that
        \begin{equation}
                r_{ij} = ({\bm a}, {\bm b})_2.
        \end{equation}
\end{proposition}
\begin{proof}
        From the recurrence of MGS, ${\bm q}_i, {\bm z}_j^{(i-1)} \in \mathcal{R}(Z)$ holds for $1\le i\le j\le n$.
        Therefore, there exist ${\bm x}, {\bm y} \in \mathcal{R}(Z)$ such that
        \begin{equation*}
                r_{ij} = ({\bm x}, {\bm y})_A = ({\bm x}, A{\bm y})_2,
        \end{equation*}
        which proves Proposition~\ref{prop:min} because $A{\bm y} \in \mathcal{R}(AZ)$.
\end{proof}
Proposition~\ref{prop:min} suggests the possibility of implementing MGS with only $n$ MV required for constructing the subspace $\mathcal{R}(AZ)$.
Therefore, we estimate the minimal computational costs for MGS to be
\begin{equation}
        \mbox{minimal costs for MGS: } n \mbox{ MV} + 2mn^2 \mbox{ flops},
        \label{eq:min}
\end{equation}
whether the column- or row-oriented is used, because the number of flops for MGS with the standard inner product is $2mn^2$.
\subsection{$n$-MV implementations of MGS: MGS-HA and MGS-HP}
\begin{algorithm}[t]
\caption{MGS-HA(col): a high accuracy type efficient implementation of MGS(col)}
\label{alg:mgs-ha}
\begin{algorithmic}[1]
        \REQUIRE $Z = [{\bm z}_1^{(0)}, {\bm z}_2^{(0)}, \dots, {\bm z}_n^{(0)}] \in \mathbb{R}^{m \times n}, A \in \mathbb{R}^{m \times m}$, where $A$ is spd.
        \ENSURE  $Q = [{\bm q}_1, {\bm q}_2, \dots, {\bm q}_n] \in \mathbb{R}^{m \times n}, R = \{ r_{ij} \}_{1 \leq i,j \leq n} \in \mathbb{R}^{n \times n}$
        \FOR{$j = 1, 2, \dots, n$}
        \FOR{$i = 1, 2, \dots, j-1$}
        \STATE $r_{ij} = {\bm p}_i^{\rm T} {\bm z}_j^{(i-1)}$
        \STATE ${\bm z}_j^{(i)} = {\bm z}_j^{(i-1)} - r_{ij} {\bm q}_i$
        \ENDFOR
        \STATE \fbox{ ${\bm x}_j^{(j-1)} = A {\bm z}_j^{(j-1)}$ }
        \STATE $r_{jj} = \sqrt{ ({\bm z}_j^{(j-1)})^{\rm T} {\bm x}_j^{(j-1)} }$
        \STATE ${\bm q}_j = {\bm z}_j^{(j-1)} / r_{jj}$
        \STATE ${\bm p}_j = {\bm x}_j^{(j-1)} / r_{jj}$
        \ENDFOR  
\end{algorithmic}
\end{algorithm}
Here, we propose two types of $n$-MV implementations of both the column- and row-oriented MGS: a high accuracy type (MGS-HA) and a high performance type (MGS-HP).
\par
Firstly, we introduce a technique to achieve $n$-MV implementations for the column-oriented MGS (Algorithm~\ref{alg:mgs-naive-col}).
In each iteration for $j$, the column-oriented MGS requires two MV.
One is for computing the $A$-norm $r_{jj} = \| {\bm z}_j^{(j-1)} \|_A$ by 
\begin{align*}
        &{\bm x}_j^{(j-1)} = A {\bm z}_j^{(j-1)}, \\
        &r_{jj} = \sqrt{ ({\bm z}_j^{(j-1)} )^{\rm T} {\bm x}_j^{(j-1)} },
\end{align*}
and another is for computing the $A$-inner product $r_{ij} = ({\bm q}_i, {\bm z}_j^{(j-1)})_A$ $(i < j)$ by
\begin{align*}
        &{\bm p}_j = A {\bm q}_j, \\
        &r_{ij} = {\bm p}_i^{\rm T} {\bm z}_j^{(i-1)} \quad (i < j).
\end{align*}
Based on these formula and ${\bm q}_j = {\bm z}_j^{(j-1)}/r_{jj}$, we can compute ${\bm p}_j = A{\bm q}_j$ without MV by
\begin{equation}
        {\bm p}_j 
        = A{\bm q}_j 
        = A \frac{ {\bm z}_j^{(j-1)} }{ r_{jj} }
        = \frac{ A {\bm z}_j^{(j-1)} }{ r_{jj} }
        = \frac{ {\bm x}_j^{(j-1)} }{ r_{jj} },
        \label{eq:tec1}
\end{equation}
which achieves an $n$-MV implementation of the column-oriented MGS as shown in Algorithm~\ref{alg:mgs-ha}.
\par
Algorithm~\ref{alg:mgs-ha} has nearly the same error bounds as MGS-naive, as we will show in Section~\ref{sec:error}.
In this sense, we call this a high accuracy type MGS, MGS-HA.
The computational cost of MGS-HA is $n \mbox{ MV} + 2mn^2 \mbox{ flops}$, which is the same as the estimated minimal computational cost \eqref{eq:min}. 
Therefore, regarding the computational cost, MGS-HA is an optimal implementation for MGS.
%
\begin{algorithm}[t]
\caption{MGS-HP(col): a high performance type efficient implementation of MGS(col)}
\label{alg:mgs-hp}
\begin{algorithmic}[1]
        \REQUIRE $Z = [{\bm z}_1^{(0)}, {\bm z}_2^{(0)}, \dots, {\bm z}_n^{(0)}] \in \mathbb{R}^{m \times n}, A \in \mathbb{R}^{m \times m}$, where $A$ is spd.
        \ENSURE  $Q = [{\bm q}_1, {\bm q}_2, \dots, {\bm q}_n] \in \mathbb{R}^{m \times n}, R = \{ r_{ij} \}_{1 \leq i,j \leq n} \in \mathbb{R}^{n \times n}$
        \STATE \fbox{ $X = [{\bm x}_1^{(0)}, {\bm x}_2^{(0)}, \dots, {\bm x}_n^{(0)}]= A Z$ }
        \FOR{$j = 1, 2, \dots, n$}
        \FOR{$i = 1, 2, \dots, j-1$}
        \STATE $r_{ij} = {\bm p}_i^{\rm T} {\bm z}_j^{(i-1)}$
        \STATE ${\bm z}_j^{(i)} = {\bm z}_j^{(i-1)} - r_{ij} {\bm q}_i$
        \ENDFOR
        \STATE ${\bm x}_j^{(j-1)} = {\bm x}_j^{(0)} - \sum_{i=1}^{j-1} r_{ij} {\bm p}_i$
        \STATE $r_{jj} = \sqrt{ ({\bm z}_j^{(j-1)})^{\rm T} {\bm x}_j^{(j-1)} }$
        \STATE ${\bm q}_j = {\bm z}_j^{(j-1)} / r_{jj}$
        \STATE ${\bm p}_j = {\bm x}_j^{(j-1)} / r_{jj}$
        \ENDFOR  
\end{algorithmic}
\end{algorithm}
\par
Although MGS-HA is optimal in terms of the computational cost, it performs one MV in each iteration of the $j$ loop.
This sequential MV reduces the computational performance.
On the other hand, Proposition~\ref{prop:min} indicates that $n$ MV can be performed together because MV are required only for constructing the subspace $\mathcal{R}(AZ)$.
In other words, we firstly compute $AZ$, then we can compute all the elements $r_{ij}$ from the matrices $Z$ and $AZ$ without MV.
\par
To achieve this, we compute ${\bm x}_j^{(i)} =  A {\bm z}_j^{(i)}$ by
\begin{equation}
        {\bm x}_j^{(i)} 
        =  A {\bm z}_j^{(i)}
        =  A \left( {\bm z}_j^{(0)} - \sum_{i=1}^{j-1} r_{ij} {\bm q}_i \right)
        =  A {\bm z}_j^{(0)} - \sum_{i=1}^{j-1} r_{ij} A{\bm q}_i
        = {\bm x}_j^{(0)} - \sum_{i=1}^{j-1} r_{ij} {\bm p}_i,
        \label{eq:tec2}
\end{equation}
where $X = [{\bm x}_1^{(0)}, {\bm x}_2^{(0)}, \dots, {\bm x}_n^{(0)}]= A Z$ as shown in Algorithm~\ref{alg:mgs-hp}.
The computational cost of Algorithm~\ref{alg:mgs-hp} is $n \mbox{ MV} + 3mn^2 \mbox{ flops}$, which is larger than that of MGS-HA.
However, Algorithm~\ref{alg:mgs-hp} is expected to show higher computational performance and smaller computational time than MGS-HA (cost: $n \mbox{ MV} + 2mn^2 \mbox{ flops}$), because a matrix-matrix multiplication is much faster than the sequential MV.
In this sense, we call this a high performance type MGS, MGS-HP.
\par
We can derive $n$-MV implementations of the row-oriented MGS in the same manner.
The vector ${\bm p}_i = A{\bm q}_i$ is computed without MV by
\begin{equation*}
        {\bm p}_i \left(= A {\bm q}_i \right) = \frac{ {\bm x}_i^{(i-1)} }{ r_{ii} },
\end{equation*}
as well as \eqref{eq:tec1} and the vector ${\bm x}_i^{(i-1)} = A{\bm z}_i^{(i-1)}$ is computed without a sequential MV by
\begin{equation*}
        {\bm x}_i^{(i-1)} \left(= A{\bm z}_i^{(i-1)}\right) = {\bm x}_i^{(0)} - \sum_{j=i+1}^{n} r_{ij} {\bm p}_i
\end{equation*}
as well as \eqref{eq:tec2} for MGS-HP(row).
The algorithms of MGS-HA(row) and MGS-HP(row) are shown in Algorithms~\ref{alg:mgs-ha-row} and \ref{alg:mgs-hp-row}, respectively.
\begin{algorithm}[t]
\caption{MGS-HA(row): a high accuracy type efficient implementation of MGS(row)}
\label{alg:mgs-ha-row}
\begin{algorithmic}[1]
        \REQUIRE $Z = [{\bm z}_1^{(0)}, {\bm z}_2^{(0)}, \dots, {\bm z}_n^{(0)}] \in \mathbb{R}^{m \times n}, A \in \mathbb{R}^{m \times m}$, where $A$ is spd.
        \ENSURE  $Q = [{\bm q}_1, {\bm q}_2, \dots, {\bm q}_n] \in \mathbb{R}^{m \times n}, R = \{ r_{ij} \}_{1 \leq i,j \leq n} \in \mathbb{R}^{n \times n}$
        \FOR{$i = 1, 2, \dots, n$}
        \STATE \fbox{ ${\bm x}_i^{(i-1)} = A {\bm z}_i^{(i-1)}$ }
        \STATE $r_{ii} = \sqrt{ ({\bm z}_i^{(i-1)})^{\rm T} {\bm x}_i^{(i-1)} }$
        \STATE ${\bm q}_i = {\bm z}_i^{(i-1)} / r_{ii}$
        \STATE ${\bm p}_i = {\bm x}_i^{(i-1)} / r_{ii}$
        \FOR{$j = i+1, i+2, \dots, n$}
        \STATE $r_{ij} = {\bm p}_i^{\rm T} {\bm z}_j^{(i-1)}$
        \STATE ${\bm z}_j^{(i)} = {\bm z}_j^{(i-1)} - r_{ij} {\bm q}_i$
        \ENDFOR
        \ENDFOR  
\end{algorithmic}
\end{algorithm}
\begin{algorithm}[t]
\caption{MGS-HP(row): a high performance type efficient implementation of MGS(row)}
\label{alg:mgs-hp-row}
\begin{algorithmic}[1]
        \REQUIRE $Z = [{\bm z}_1^{(0)}, {\bm z}_2^{(0)}, \dots, {\bm z}_n^{(0)}] \in \mathbb{R}^{m \times n}, A \in \mathbb{R}^{m \times m}$, where $A$ is spd.
        \ENSURE  $Q = [{\bm q}_1, {\bm q}_2, \dots, {\bm q}_n] \in \mathbb{R}^{m \times n}, R = \{ r_{ij} \}_{1 \leq i,j \leq n} \in \mathbb{R}^{n \times n}$
        \STATE \fbox{ $X = [{\bm x}_1^{(0)}, {\bm x}_2^{(0)}, \dots, {\bm x}_n^{(0)}]= A Z$ }
        \FOR{$i = 1, 2, \dots, n$}
        \STATE $r_{ii} = \sqrt{ ({\bm z}_i^{(i-1)})^{\rm T} {\bm x}_i^{(i-1)} }$
        \STATE ${\bm q}_i = {\bm z}_i^{(i-1)} / r_{ii}$
        \STATE ${\bm p}_i = {\bm x}_i^{(i-1)} / r_{ii}$
        \FOR{$j = i+1, i+2, \dots, n$}
        \STATE $r_{ij} = {\bm p}_i^{\rm T} {\bm z}_j^{(i-1)}$
        \STATE ${\bm z}_j^{(i)} = {\bm z}_j^{(i-1)} - r_{ij} {\bm q}_i$
        \ENDFOR
        \STATE ${\bm x}_i^{(i-1)} = {\bm x}_i^{(0)} - \sum_{j=i+1}^{n} r_{ij} {\bm p}_i$
        \ENDFOR  
\end{algorithmic}
\end{algorithm}
\par
The proposed concept can also be applied to CGS for its $n$-MV implementations: CGS-HA(col/row) and CGS-HP(col/row).
It is also noted that the HP-type of row-oriented versions: MGS-HP(row) and CGS-HP(row), are equivalent to the algorithms introduced in \cite{Dubrulle:2001} to use in the block conjugate gradient method for solving linear systems with multiple right-hand sides.
However, the performance of these algorithms are not analyzed and evaluated in \cite{Dubrulle:2001}, because the main objective of \cite{Dubrulle:2001} is to propose the block conjugate gradient method.
\section{Analysis of error bounds}
\label{sec:error}
In this section, we present error bounds on the representation error and the loss of $A$-orthogonality of MGS-HA (Algorithm~\ref{alg:mgs-ha}) and show that MGS-HA has nearly the same error bounds as MGS-naive (Algorithm~\ref{alg:mgs-naive-col}).
\par
Let $\alpha \in \mathbb{R}, {\bm x} \in \mathbb{R}^{m}, A \in \mathbb{R}^{m \times n}$ and let $\widehat\alpha \in \mathbb{R}, \widehat{\bm x} \in \mathbb{R}^{m}, \widehat{A} \in \mathbb{R}^{m \times n}$ denote their counterparts computed in floating-point arithmetic. 
Also, we denote by $|A|$ and $|{\bm x}|$ the matrix and the vector whose entries are absolute values of entries of $A$ and ${\bm x}$, respectively.
\par
Assuming that $\alpha \in \mathbb{R}, {\bm x}, {\bm y} \in \mathbb{R}^{m}, A \in \mathbb{R}^{m \times m}$, we use the following error bounds for scaling ${\bm y} = \alpha{\bm x}$, inner product $\alpha = {\bm x}^{\rm T} {\bm y}$ and MV ${\bm y} = A {\bm x}$ computed in floating-point arithmetic:
\begin{align}
        &\widehat{\bm y} = \alpha {\bm x} + \Delta {\bm y}, \quad | \Delta {\bm y} | \leq {\bf u} | \alpha | | {\bm x} |, \label{eq:error_AXPY} \\
        &\widehat{\alpha} = {\bm x}^{\rm T} {\bm y} + \Delta \alpha, \quad | \Delta \alpha | \leq \gamma_m | {\bm x} |^{\rm T} | {\bm y} |, \label{eq:error_inner} \\
        &\widehat{\bm y} = A {\bm x} + \Delta {\bm y}, \quad | \Delta {\bm y} | \leq \gamma_m |A| | {\bm x} |, \label{eq:error_MV}
\end{align}
where ${\bf u}$ is the unit rounding error and $\gamma_m := m {\bf u} / (1 - m {\bf u}) \approx m {\bf u}$ \cite{Higham:2002}.
\subsection{Upper bound of representation error}
The recurrence formulas of ${\bm z}_j^{(i)}$ and ${\bm q}_j$ in Gram-Schmidt orthogonalization are written as
\begin{align}
        &{\bm z}_j^{(i)} = {\bm z}_j^{(i-1)}-r_{ij}{\bm q}_i \quad (i=1, 2, \ldots, j-1), \label{eq:AXPY} \\
        &{\bm q}_j = \frac{{\bm z}_j^{(j-1)}}{r_{jj}}. \label{eq:scaling}
\end{align}
These formulas are independent of the inner product used.
They are also the same whether the naive implementation (MGS-naive, Algorithm~\ref{alg:mgs-naive-col}) or the proposed implementation (MGS-HA, Algorithm~\ref{alg:mgs-ha}) is used.
The only difference between MGS-naive and MGS-HA lies in how to compute $r_{ij}$.
\par
In \cite[Theorem 3.1]{Rozloznik:2012}, an upper bound on the representation error of MGS-naive in floating-point arithmetic is derived as
\begin{equation}
        \|Z - \widehat{Q} \widehat{R} \| \le O(n^{3/2})\left(\|Z\|+\|\widehat{Q}\| \|\widehat{R}\|\right)
        \label{eq:Th3.1}
\end{equation}
based only on \eqref{eq:AXPY} and \eqref{eq:scaling}. 
Because MGS-HA also uses \eqref{eq:AXPY} and \eqref{eq:scaling}, we have the same upper bound on the representation error of MGS-HA.
It is to be noted that the upper bound (\ref{eq:Th3.1}) depends on the computed results $\widehat{Q}, \widehat{R}$, so it is an {\it a posteriori} error bound.
Hence, \eqref{eq:Th3.1} means that, if the norms of the computed results $\widehat{Q}, \widehat{R}$ are nearly the same for both methods, they have nearly the same upper bounds.
\par
Eqs.~\eqref{eq:AXPY} and \eqref{eq:scaling} are also the same for CGS-naive and CGS-HA.
Therefore, we have the same upper bound of the representation error for CGS-naive and CGS-HA.
\subsection{Upper bound of loss of $A$-orthogonality}
The main difference between MGS-naive and MGS-HA lies in how to compute $r_{ij}$ for the strict upper triangular part ($i < j$), because both of the methods compute the diagonal element by $r_{jj}=({\bm z}_j^{(j-1)})^{\rm T}A{\bm z}_j^{(j-1)}$.
\par
MGS-naive computes $r_{ij}$ from ${\bm q}_i$ and ${\bm z}_j^{(i-1)}$ $(i<j)$ by
%
\begin{align}
        &{\bm p}_i = A{\bm q}_i, \label{eq:conv1} \\
        &r_{ij} = {\bm p}_i^{\rm T}{\bm z}_j^{(i-1)}. \label{eq:conv2}
\end{align}
In contrast, MGS-HA computes $r_{ij}$ from the unnormalized vector ${\bm z}_j^{(j-1)}$ by
\begin{align}
        &{\bm x}_i^{(i-1)} = A{\bm z}_i^{(i-1)} \label{eq:Imakura1}, \\
        &{\bm p}_i = \frac{{\bm x}_i^{(i-1)}}{r_{ii}} \label{eq:Imakura2}, \\
        &r_{ij} = {\bm p}_i^{\rm T}{\bm z}_j^{(i-1)} \label{eq:Imakura3}.
\end{align}
On the other hand, the vector ${\bm q}_i$ is computed by normalization of ${\bm z}_i^{(i-1)}$ as in MGS-naive, i.e.,
\begin{equation*}
        {\bm q}_i = \frac{{\bm z}_i^{(i-1)}}{r_{ii}}.
\end{equation*}
\par
According to \cite[Theorem 3.2]{Rozloznik:2012}, the local errors of $A$-inner product, AXPY \eqref{eq:AXPY} and scaling \eqref{eq:scaling} are propagated by $\widehat{R}^{-1}$ to be the loss of $A$-orthogonality of MGS-naive, $\widehat{Q}^{\rm T}A\widehat{Q}-I_n$.
Eqs.~\eqref{eq:AXPY} and \eqref{eq:scaling} are same for both methods.
We can use the same evaluation for the norm of $\widehat{R}^{-1}$.
Therefore, we just analyze the local error of the $A$-inner product.
\par
From the error bounds of MV and inner product, (\ref{eq:error_MV}) and (\ref{eq:error_inner}), Eqs.~\eqref{eq:conv1} and \eqref{eq:conv2} in floating-point arithmetic can be written as
\begin{align}
        &\widehat{\bm p}_i = A\widehat{\bm q}_i + \Delta{\bm p}_i, \quad
        |\Delta{\bm p}_i| \le \gamma_m|A||\widehat{\bm q}_i|, \label{eq:conv1_f} \\
        &\widehat r_{ij} = \widehat{\bm p}_i^{\rm T}\widehat{\bm z}_j^{(i-1)} + \Delta r_{ij}, \quad 
        |\Delta r_{ij}|\le \gamma_m|\widehat{\bm p}_i|^{\rm T}|\widehat{\bm z}_j^{(i-1)}|. \label{eq:conv2_f}
\end{align}
From \eqref{eq:conv1_f} and \eqref{eq:conv2_f}, an error bound of $\widehat r_{ij}$ computed by MGS-naive, ignoring terms of $\mathcal{O}({\bf u}^2)$, is derived \cite{Rozloznik:2012} as
\begin{align}
        |\widehat r_{ij} - \widehat{\bm q}_i^{\rm T}A\widehat{\bm z}_j^{(i-1)} |
        & = |( \widehat{\bm p}_i- A\widehat{\bm q}_i )^{\rm T} \widehat{\bm z}_j^{(i-1)}  + \Delta r_{ij}  | \nonumber \\
        &\le |(\Delta{\bm p}_i)^{\rm T}\widehat{\bm z}_j^{(i-1)}| + |\Delta r_{ij}| \nonumber \\
        &\le \|\Delta{\bm p}_i \| \|\widehat{\bm z}_j^{(i-1)} \| + \gamma_m \|\widehat{\bm p}_i \| \| \widehat{\bm z}_j^{(i-1)} \| \nonumber \\
        &\le \gamma_m\|\,|A|\,\| \|\widehat{\bm q}_i \| \|\widehat{\bm z}_j^{(i-1)}\| + \gamma_m\|A\| \|\widehat{\bm q}_i\| \|\widehat{\bm z}_j^{(i-1)}\| \nonumber \\
        &\le \gamma_{m\sqrt{m}+m}\|A\| \|\widehat{\bm q}_i\| \|\widehat{\bm z}_j^{(i-1)}\|,
        \label{eq:conv3_f}
\end{align}
where we used $\|\,|A|\,\| \le \|A\|_{\rm F} \le \sqrt{m}\|A\|$, $\ell\gamma_k\le\gamma_{\ell k}$ and $\gamma_k+\gamma_{\ell}\le\gamma_{k+\ell}$ \cite{Higham:2002}.
\par
In contrast, formulas \eqref{eq:Imakura1}--\eqref{eq:Imakura3} of MGS-HA in floating-point arithmetic become
\begin{align}
        &\widehat{\bm x}_i^{(i-1)} = A\widehat{\bm z}_i^{(i-1)} + \Delta{\bm x}_i^{(i-1)}, \quad
        |\Delta{\bm x}_i^{(i-1)}| \le \gamma_m|A||\widehat{\bm z}_i^{(i-1)}|, \label{eq:Imakura1_f} \\
        &\widehat{\bm p}_i = \frac{\widehat{\bm x}_i^{(i-1)}}{ \widehat r_{ii}} + \Delta{\bm p}_i, \quad 
        |\Delta{\bm p}_i| \le {\bf u} \frac{|\widehat{\bm x}_i^{(i-1)}|}{|\widehat r_{ii}|}, \label{eq:Imakura2_f} \\
        &\widehat r_{ij} = \widehat{\bm p}_i^{\rm T}\widehat{\bm z}_j^{(i-1)} + \Delta r_{ij}, \quad |\Delta r_{ij}| \le \gamma_m|\widehat{\bm p}_i|^{\rm T} |\widehat{\bm z}_j^{(i-1)}|. \label{eq:Imakura3_f}
\end{align}
These formulas compute $r_{ij}$ from $\widehat{\bm z}_i^{(i-1)}$ and $\widehat{\bm z}_j^{(i-1)}$.
Because the local error of $A$-inner product is defined as the difference between $\widehat r_{ij}$ and $\widehat{\bm q}_i^{\rm T}A\widehat{\bm z}_j^{(i-1)}$, we also need a relationship between $\widehat{\bm q}_i$ and $\widehat{\bm z}_i^{(i-1)}$, i.e.,
\begin{equation}
        \widehat{\bm q}_i = \frac{ \widehat{\bm z}_i^{(i-1)} }{\widehat r_{ii}} + \Delta{\bm q}_i, \quad
        |\Delta{\bm q}_i| \le {\bf u}\frac{|\widehat{\bm z}_i^{(i-1)}|}{|\widehat r_{ii}|}. \label{eq:Imakura4_f}
\end{equation}
Substituting \eqref{eq:Imakura3_f}, \eqref{eq:Imakura2_f}, \eqref{eq:Imakura1_f} and \eqref{eq:Imakura4_f} into $|\widehat r_{ij} - \widehat{\bm q}_i^{\rm T}A\widehat{\bm z}_j^{(i-1)}|$ in this order and ignoring terms of $\mathcal{O}({\bf u}^2)$, we have
\begin{align}
        &|\widehat r_{ij} - \widehat{\bm q}_i^{\rm T}A\widehat{\bm z}_j^{(i-1)}| \nonumber \\
        &\quad \le 
        \left|\frac{(\Delta{\bm x}_i^{(i-1)})^{\rm T}\widehat{\bm z}_j^{(i-1)}}{\widehat r_{ii}}\right|
        + \left|(\Delta{\bm q}_i)^{\rm T}A\widehat{\bm z}_j^{(i-1)} \right|
        + \left|\Delta{\bm p}_i^{\rm T}\widehat{\bm z}_j^{(i-1)} \right|
        + |\Delta r_{ij}| \nonumber \\
        &\quad \le 
        \gamma_m \frac{\|\,|A|\,\| \|\widehat{\bm z}_i^{(i-1)}\| \|\widehat{\bm z}_j^{(i-1)}\|}{|\widehat r_{ii}|}
        + {\bf u} \frac{\|A\| \|\widehat{\bm z}_i^{(i-1)}\| \|\widehat{\bm z}_j^{(i-1)}\|}{|\widehat r_{ii}|}
        + {\bf u} \frac{\|\widehat{\bm x}_i^{(i-1)}\| \|\widehat{\bm z}_j^{(i-1)}\|}{|\widehat r_{ii}|} \nonumber \\
        &\quad \hphantom{\le} \quad + \gamma_m \|\widehat{\bm p}_i\| \|\widehat{\bm z}_j^{(i-1)}\| \nonumber \\
        &\quad \le 
        \gamma_m\|\,|A|\,\| \|\widehat{\bm q}_i\| \|\widehat{\bm z}_j^{(i-1)}\|
        + {\bf u}\|A\| \|\widehat{\bm q}_i\| \|\widehat{\bm z}_j^{(i-1)}\|
        + {\bf u}\|A\| \|\widehat{\bm q}_i\| \|\widehat{\bm z}_j^{(i-1)}\| \nonumber \\
        &\quad \hphantom{\le} \quad + \gamma_m\|A\| \|\widehat{\bm q}_i\| \|\widehat{\bm z}_j^{(i-1)}\| \nonumber \\
        &\quad \le \gamma_{m\sqrt{m}+m+2} \|A\| \|\widehat{\bm q}_i\| \|\widehat{\bm z}_j^{(i-1)}\|.
\label{eq:Imakura5_f}
\end{align}
\par
Comparing \eqref{eq:conv3_f} for MGS-naive and \eqref{eq:Imakura5_f} for MGS-HA, we know that the only difference is in the coefficients:
\begin{align*}
        &\gamma_{m\sqrt{m}+m} \approx (m \sqrt{m} + m) {\bf u} \approx \mathcal{O}(m^{3/2}) {\bf u}, \\
        &\gamma_{m\sqrt{m}+m+2} \approx (m \sqrt{m} + m + 2) {\bf u} \approx \mathcal{O}(m^{3/2}) {\bf u}.
\end{align*}

In \cite{Rozloznik:2012}, it is shown that the strict upper triangular part of the loss of $A$-orthogonality $\widehat{Q}^{\rm T}A\widehat{Q}-I_n$, which is denoted as $\Delta E^{(3)}$, can be bounded as
\begin{equation}
        \|\Delta E^{(3)}\| \le \|\widehat{R}^{-1}\| \|\Delta E^{(2)}\|_{\rm F}, \label{eq:E3bound}
\end{equation}
where $\Delta E^{(2)}$ is a strict upper triangular matrix defined by
\begin{align}
&[\Delta E^{(2)}]_{ij} = -(\widehat{\bm q}_i,\Delta{\bm y}_j^{(i)})_A+(\widehat{\bm q}_i,\sum_{k=i}^j\Delta{\bm d}_j^{(k)})_A, \label{eq:E2definition} \\
&\Delta{\bm y}_j^{(i)} = (\widehat{r}_{ij}-\widehat{\bm q}_i^{\rm T}A\widehat{\bm z}_j^{(i-1)})\widehat{\bm q}_i+(\|\widehat{\bm q}_i\|_A^2-1)\widehat{\bm z}_j^{(i-1)} \label{eq:Deltaydefinition}
\end{align}
and $\Delta{\bm d}_j^{(i)}$ ($i<j$) and $\Delta{\bm d}_j^{(j)}$ are floating-point errors arising in the AXPY operation (\ref{eq:AXPY}) and scaling (\ref{eq:scaling}), respectively\footnote{The definition of $[\Delta E^{(2)}]_{ij}$ given in \cite{Rozloznik:2012} is actually the definition of $[\Delta E^{(2)}]_{ji}$. We corrected this in Eq.~(\ref{eq:E2definition}).}. See the proof of Theorem 3.2 in \cite{Rozloznik:2012} for details.

In Eqs.~(\ref{eq:E3bound})--(\ref{eq:Deltaydefinition}), the AXPY error $\Delta{\bm d}_j^{(i)}$ ($i<j$) and the scaling errors $\Delta{\bm d}_j^{(j)}$ and $\|\widehat{\bm q}_i\|_A^2-1$ can be bounded by the same expression in both methods, because their computational formulas are the same. The norm $\|\widehat{R}^{-1}\|$ can also be bounded in the same way in both methods. Hence, the only difference lies in the evaluation of the local error of $\widehat{r}_{ij}$, defined as $\widehat{r}_{ij}-\widehat{\bm q}_i^{\rm T}A\widehat{\bm z}_j^{(i-1)}$. But comparing (\ref{eq:conv3_f}) with (\ref{eq:Imakura5_f}) reveals that the difference in this part is slight. In addition, the diagonal part of $\widehat{Q}^{\rm T}A\widehat{Q}-I_n$ is nothing but the scaling error and has the same bound for both methods. Thus we can conclude that MGS-HA has the same {\it a posteriori} bound for the loss of $A$-orthogonality as MGS-naive \cite{Rozloznik:2012}:
\begin{align}
        \| \widehat Q^{\rm T} A \widehat{Q} - I_n \| 
        &\leq
        \frac{\mathcal{O}(m^{3/2}){\bf u}\|A\|\|\widehat{Q}\|\max_{i\le j}\frac{\|\widehat{\bm z}_j^{(i-1)}\|}{\|\widehat{\bm z}_j^{(i-1)}\|_A}\kappa(A^{1/2}Z)}{1-\mathcal{O}(m^{3/2}){\bf u}\|A\|\|\widehat{Q}\|\max_{i\le j}\frac{\|\widehat{\bm z}_j^{(i-1)}\|}{\|\widehat{\bm z}_j^{(i-1)}\|_A}\kappa(A^{1/2}Z)}
\nonumber \\
        &\leq
        \frac{\mathcal{O}(m^{3/2}){\bf u}\kappa(A)\kappa(A^{1/2}Z)}{1-\mathcal{O}(m^{3/2}){\bf u}\kappa(A)\kappa(A^{1/2}Z)},
        \label{eq:bound_a}
\end{align}
provided that $\mathcal{O}(m^{3/2}){\bf u}\kappa(A)\kappa(A^{1/2}Z)<1$.
\subsection{Analysis of CGS}
For a variant of CGS, CGS-P \cite{Smoktunowicz:2006}, that computes the diagonal element $r_{jj}$ in a different way from the original CGS, error bounds for a non-standard inner product are given in \cite{Rozloznik:2012}.
On the other hand, error bounds of original CGS have not been well analyzed yet for a non-standard inner product.
\par
However, we can estimate the influence of the proposed approach on the error bounds of CGS.
As in the case of MGS, the only difference between CGS-naive and CGS-HA is how to compute $r_{ij}$ $(i < j)$.
For both CGS-naive and CGS-HA, the recurrence formulas are obtained from those of MGS-naive and MGS-HA, respectively, by changing $\widehat{\bm z}_j^{(i-1)}$ to $\widehat{\bm z}_j^{(0)}$. 
Therefore, the local error in the computation of $\widehat r_{ij}$ can be evaluated by \eqref{eq:conv3_f} and \eqref{eq:Imakura5_f} by changing $\widehat{\bm z}_j^{(i-1)}$ to $\widehat{\bm z}_j^{(0)}$.
Thus, the local errors of $\widehat r_{ij}$ are nearly the same for both CGS-naive and CGS-HA.
As a result, we can expect that CGS-HA has nearly the same loss of $A$-orthogonality as CGS-naive.
\section{Numerical experiments}
\label{sec:experiment}
In this section, we evaluate the computational performance of MGS-HA (Algorithm~\ref{alg:mgs-ha}) and MGS-HP (Algorithm~\ref{alg:mgs-hp}). In particular, we compare the computation time and the loss of $A$-orthogonality of these methods with those of MGS-naive (Algorithm~\ref{alg:mgs-naive-col}), CGS-naive and Cholesky QR, the last of which is one of the fastest algorithms for \eqref{eq:oqr}.
\subsection{Numerical experiment I}
\begin{figure}[t]
\centering
\subfloat[Dense problem]{
\includegraphics[bb = 0 0 360 216, scale=0.5]{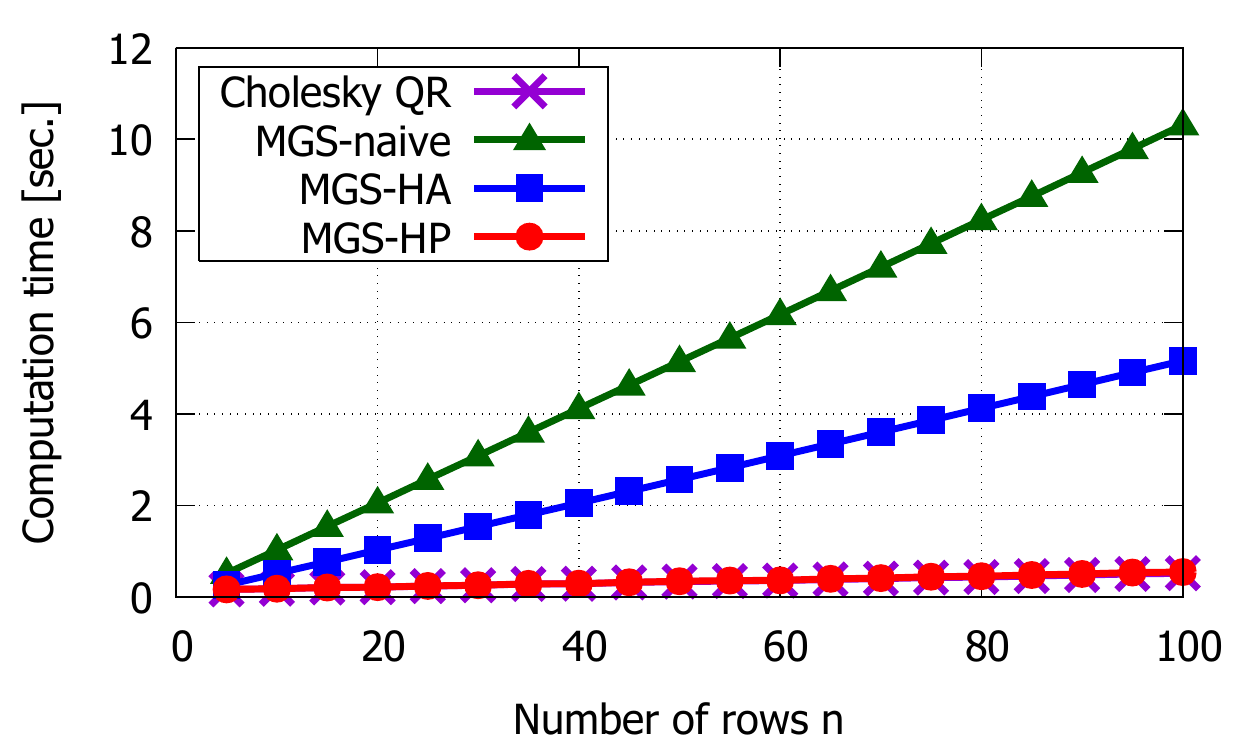}
\includegraphics[bb = 0 0 360 216, scale=0.5]{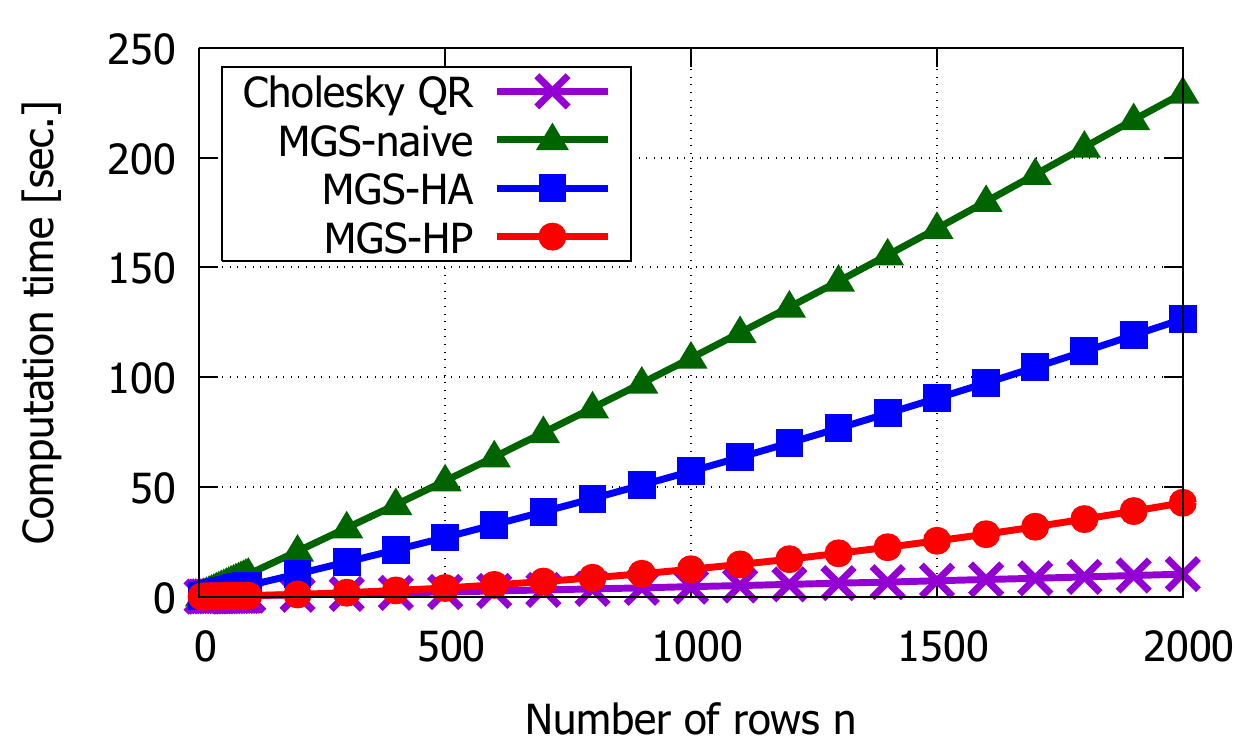}
}\\
\subfloat[Sparse problem (AUNW9180)]{
\includegraphics[bb = 0 0 360 216, scale=0.5]{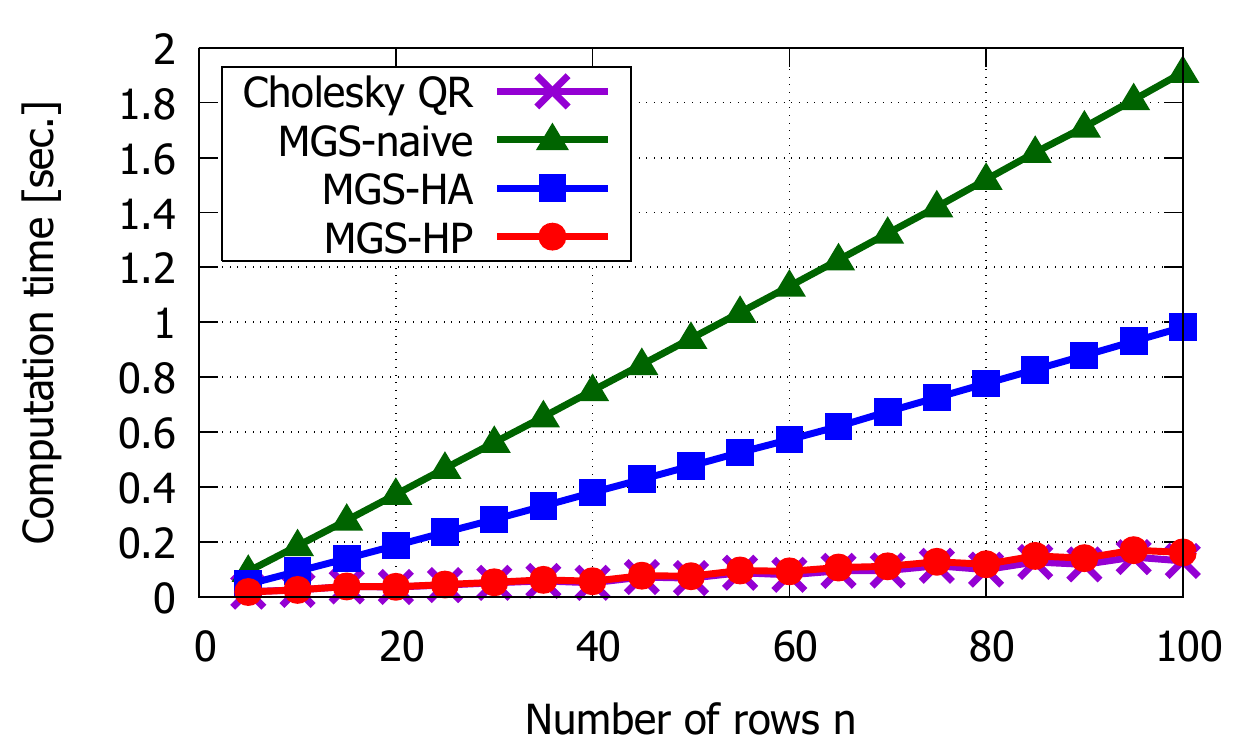}
\includegraphics[bb = 0 0 360 216, scale=0.5]{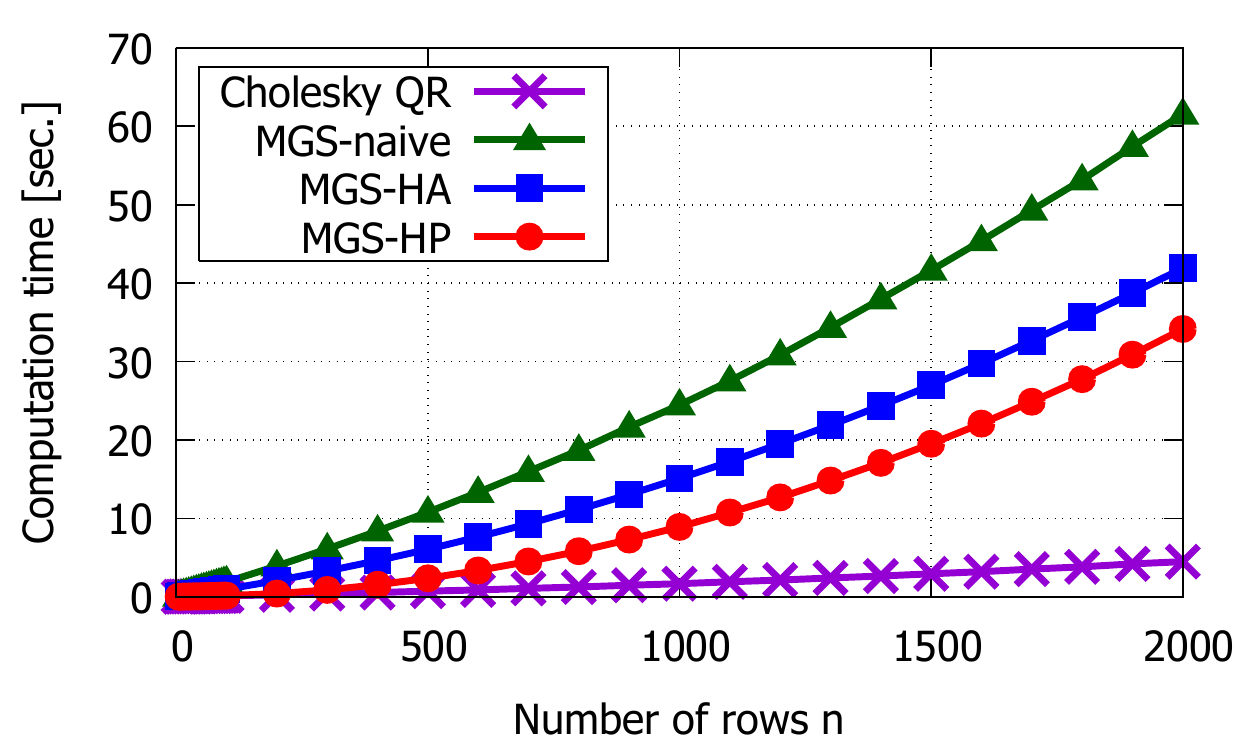}
}
\caption{
        Computation time [sec.] of MGS-naive, MGS-HA, MGS-HP and CholeskyQR.
}
\label{fig:time}
\end{figure}
\begin{figure}[t]
\centering
\subfloat[Dense problem]{
\includegraphics[bb = 0 0 360 216, scale=0.5]{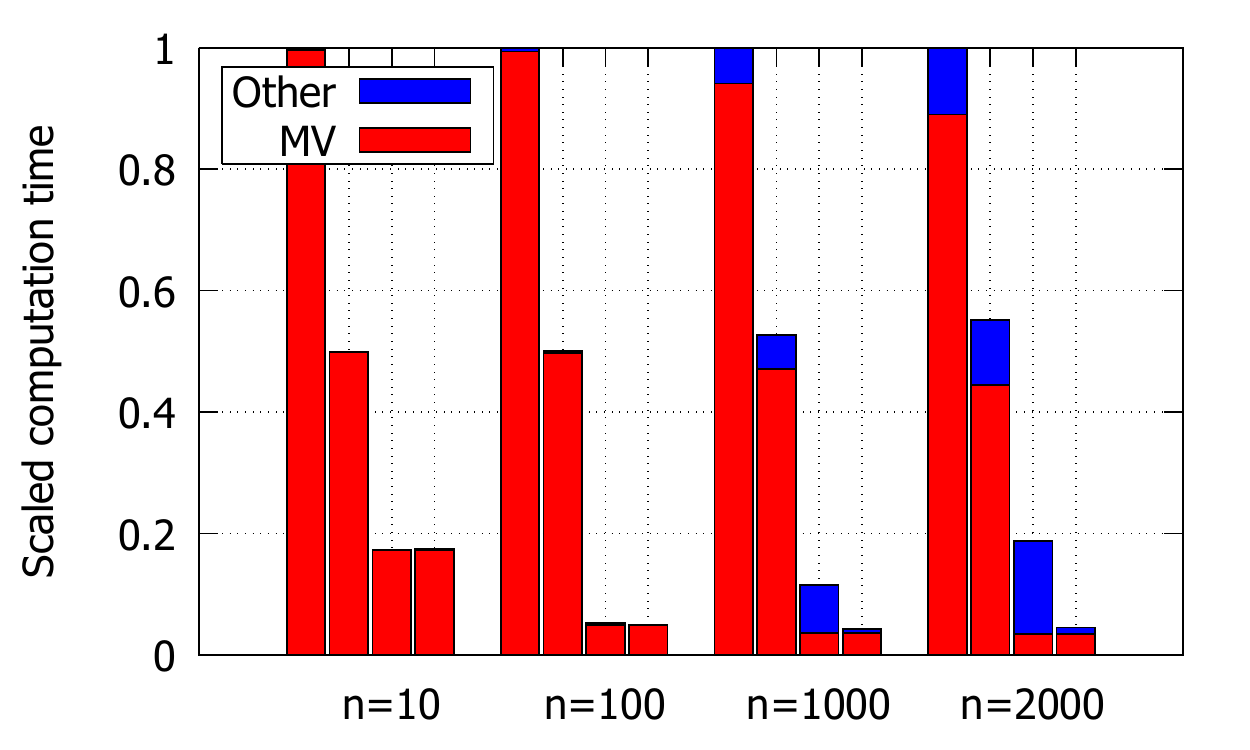}
}
\subfloat[Sparse problem (AUNW9180)]{
\includegraphics[bb = 0 0 360 216, scale=0.5]{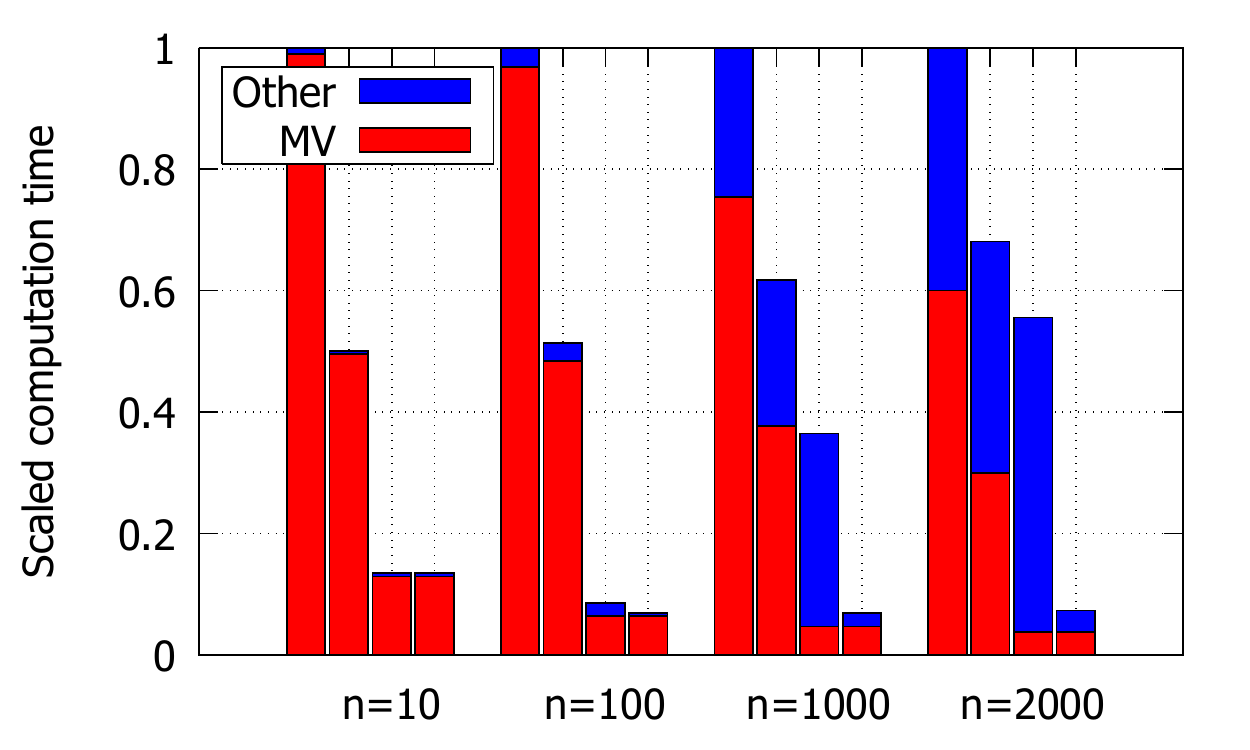}
}
\caption{
        Computation time scaled by the total computation time of MGS-naive for each $n$.
        The bar graph for each $n$ represents the time for MGS-naive, MGS-HA, MGS-HP and Cholesky QR, respectively (from left to right).
}
\label{fig:rate}
\end{figure}
Firstly, we compare the computation time of MGS-naive, MGS-HA, MGS-HP and Cholesky QR for two different problems.
For the first problem, $A$ is a random dense spd matrix with $m = 10000$.
For the second problem, $A$ is a sparse spd matrix AUNW9180 obtained from ELSES matrix library \cite{ELSES}.
This is an overlap matrix in an electronic structure calculation of a helical multishell gold nanowire.
The size of the matrix is $m = 9180$ and the number of non-zero entries is $nnz = 3557446$.
For both problems, we set $Z$ to be a random dense matrix.
We test $n = 5, 10, \dots, 100, 200, \dots, 2000$.
\par
All the numerical experiments were carried out in double precision arithmetic on OS: CentOS 64bit, CPU: Intel Xeon CPU E5-2667 3.20GHz (1 core), Memory: 48GB.
We used Intel MKL for matrix computations and Mersenne twister for generating random matrices.
\par
Figure~\ref{fig:time} shows the computation time for both problems, while Figure~\ref{fig:rate} shows breakdown of the computation time scaled by the total computation time of MGS-naive for each $n$.
When $n \ll m$, most of the computation time is used for computing MV and hence the total time increases proportionally to $n$; see the left columns of Figure~\ref{fig:time} and Figure~\ref{fig:rate}.
In this situation, MGS-HA achieves 2x speedup over MGS-naive. MGS-HP and Cholesky QR are even faster and show drastic speedup over these methods.
On the other hand, as $n$ becomes larger, the ratio of computation time for other parts increases, especially for the sparse problem.
In this situation, the speedup ratio of the proposed methods becomes relatively small, although both methods are still faster than MGS-naive.
%

%
\subsection{Numerical experiment II}
Next, we compare the loss of $A$-orthogonality
\begin{equation}
        \| \widehat{Q}^{\rm T} A \widehat{Q} - I_n \|
\end{equation}
of MGS-naive, MGS-HA, MGS-HP, CGS-naive and Cholesky QR.
Let $V \in \mathbb{R}^{m \times m}$ be a random orthogonal matrix.
Then, we set $A$ as
\begin{equation*}
        A = V D V^{\rm T},
\end{equation*}
where
\begin{equation*}
        D = {\rm diag}(d_{1}, d_{2}, \dots, d_{m}), \quad
        d_{i} = 10^{\alpha(i-1)}, \quad
        \alpha = \frac{\log_{10} \kappa(A)}{m-1},
\end{equation*}
so that $\log_{10} d_{i}$ are evenly spaced.
Also, let $W \in \mathbb{R}^{n \times n}$ be a random orthogonal matrix and $U_1, U_2 \in \mathbb{R}^{m \times n}$ be matrices whose columns are eigenvectors of $A$ corresponding to the $n$ largest and the $n$ smallest eigenvalues, respectively.
Then, we set $Z$ as
\begin{align*}
        &\mbox {case 1: } Z = U_1 E W^{\rm T}, \\
        &\mbox {case 2: } Z = U_2 E W^{\rm T},
\end{align*}
where
\begin{equation*}
        E = {\rm diag}(e_{1}, e_{2}, \dots, e_{n}), \quad
        e_{i} = 10^{\beta(i-1)}, \quad
        \beta = \frac{\log_{10} \kappa(Z)}{n-1}.
\end{equation*}
Case 1 and case 2 provide a best case and a worst case with respect to the loss of $A$-orthogonality, respectively \cite{Lowery:2014}.
We set $m=100, n=20$ and test $28^2$ problems with $\kappa(A), \kappa(A^{1/2} Z) = 10^{0.5}, 10^1, 10^{1.5}, \dots, 10^{14}$ for each case.
\par
All the numerical experiments were carried out in MATLAB2016a.
We used Mersenne twister for generating random matrices.
\begin{figure}[!t]
\centering
\subfloat[MGS-naive]{
\includegraphics[bb = 0 180 595 642, scale=0.30]{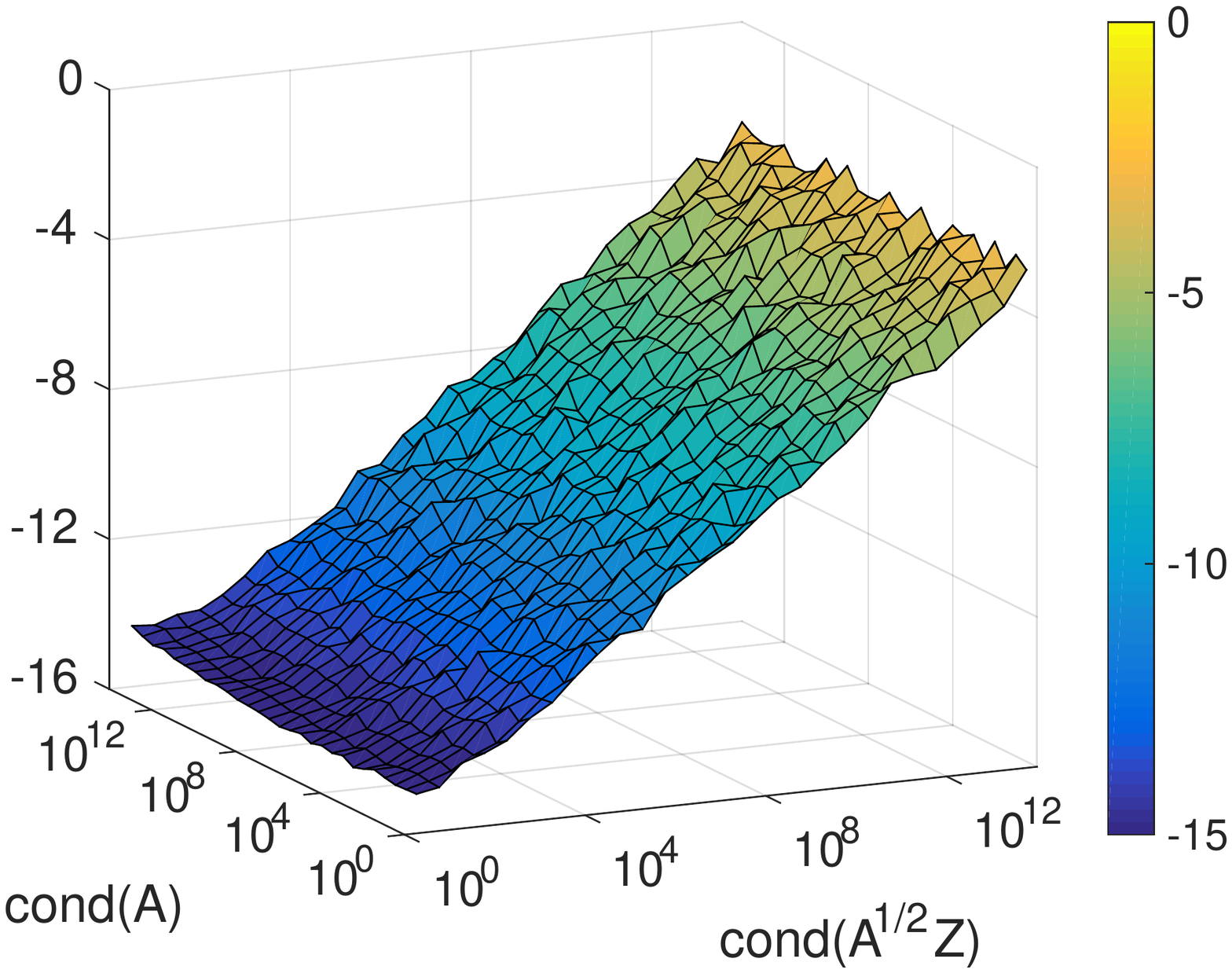}
}\\
\subfloat[MGS-HA]{
\includegraphics[bb = 0 180 595 642, scale=0.30]{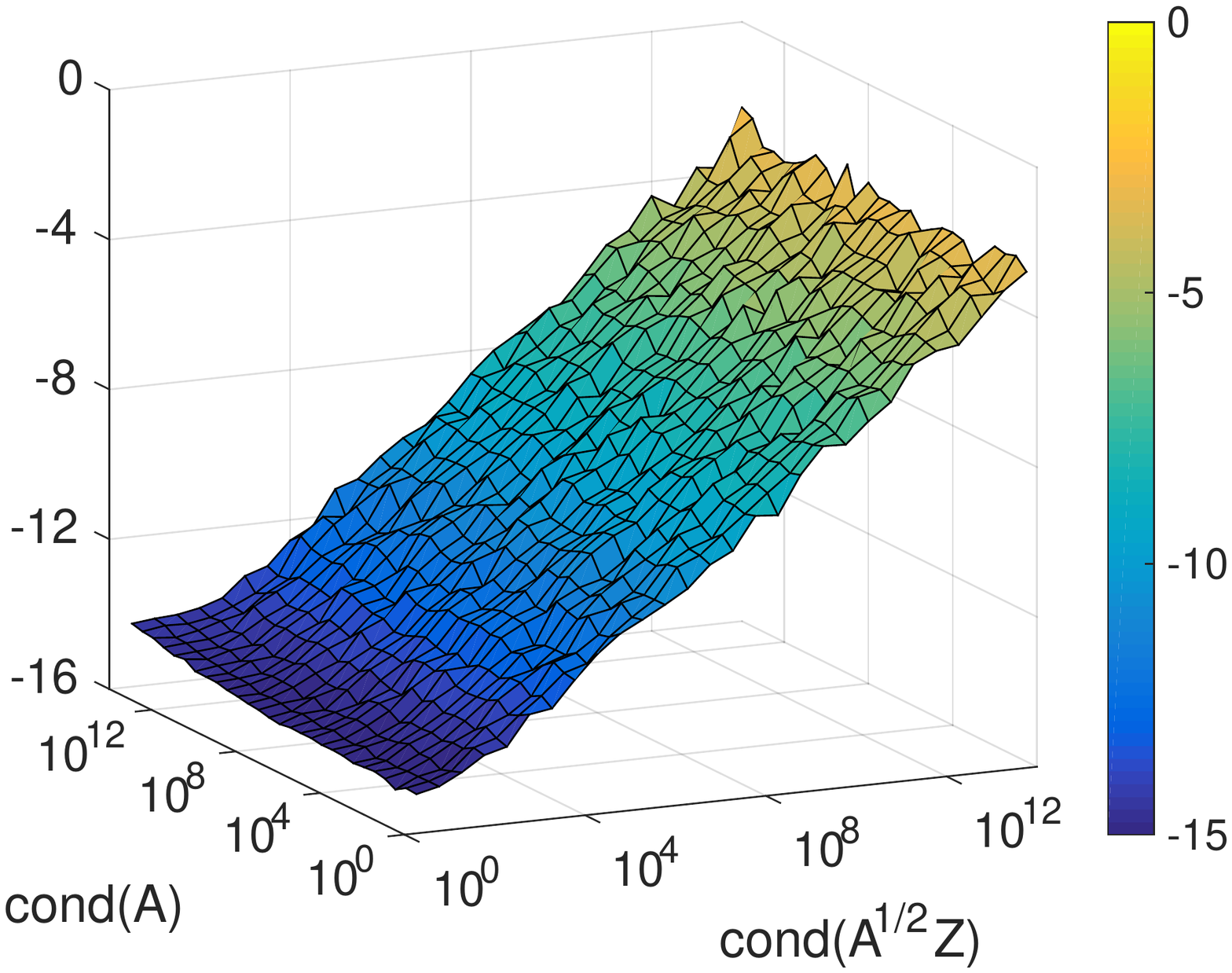}
}
\subfloat[MGS-HP]{
\includegraphics[bb = 0 180 595 642, scale=0.30]{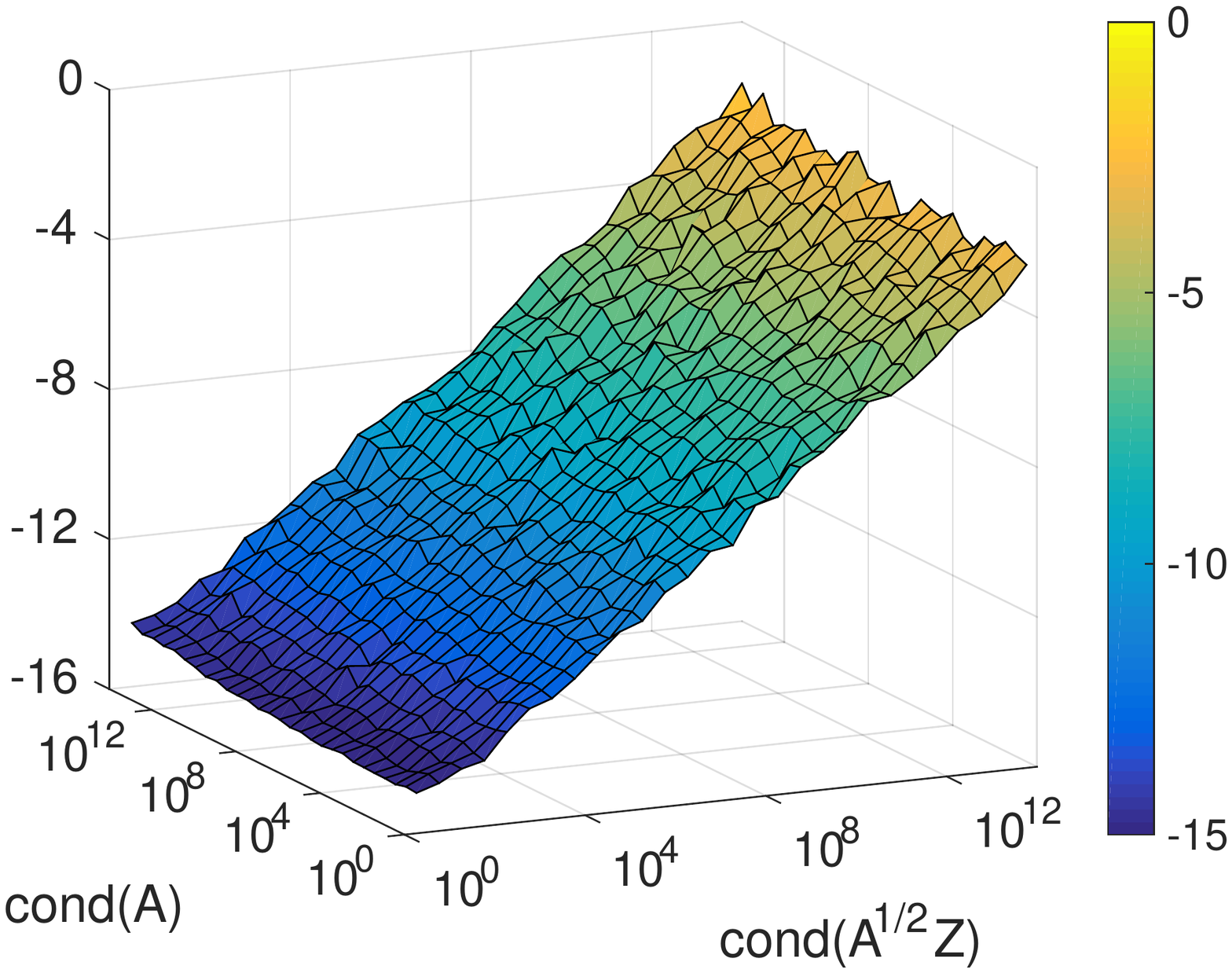}
}\\
\subfloat[Cholesky QR]{
\includegraphics[bb = 0 180 595 642, scale=0.30]{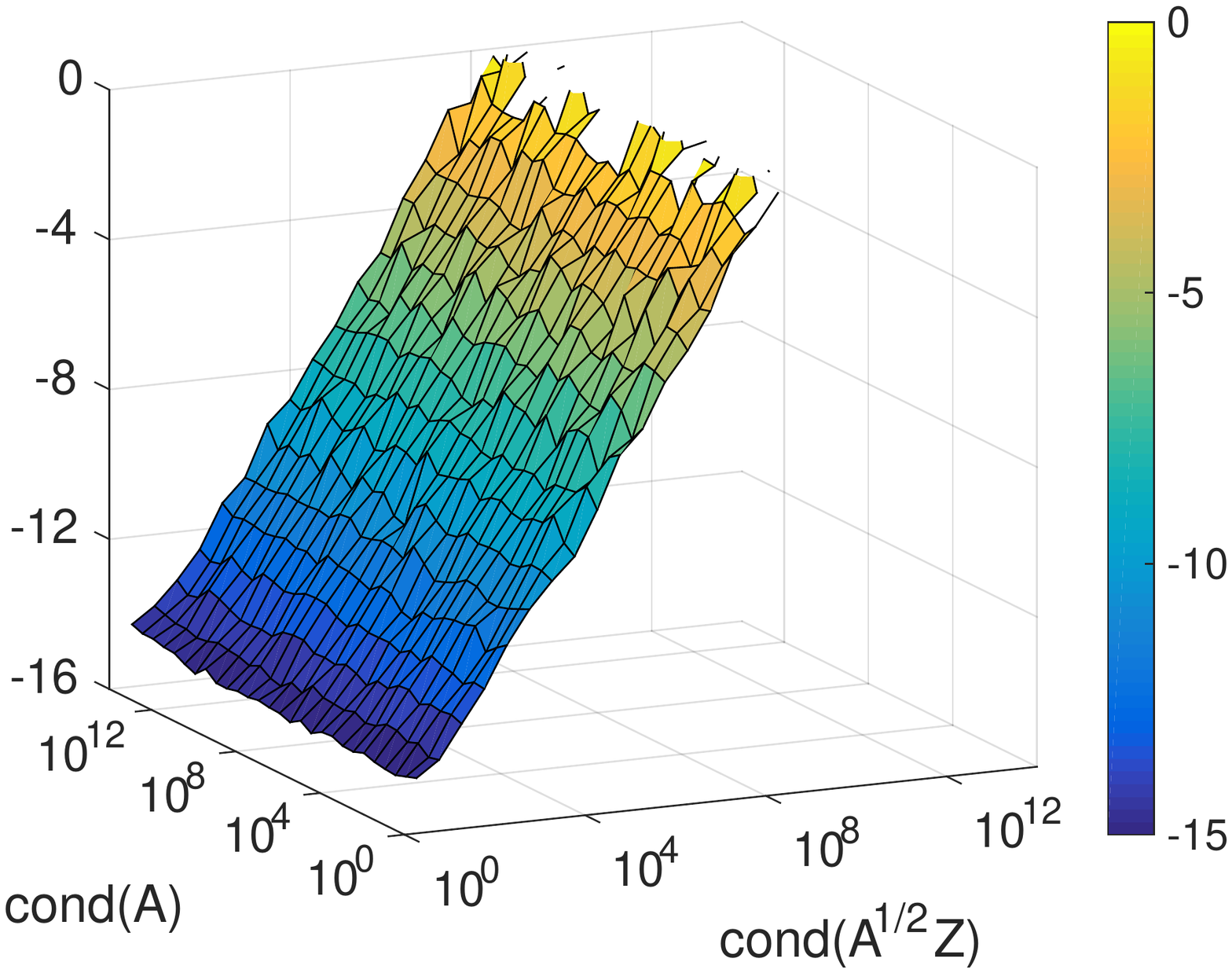}
}
\subfloat[CGS-naive]{
\includegraphics[bb = 0 180 595 642, clip, scale=0.30]{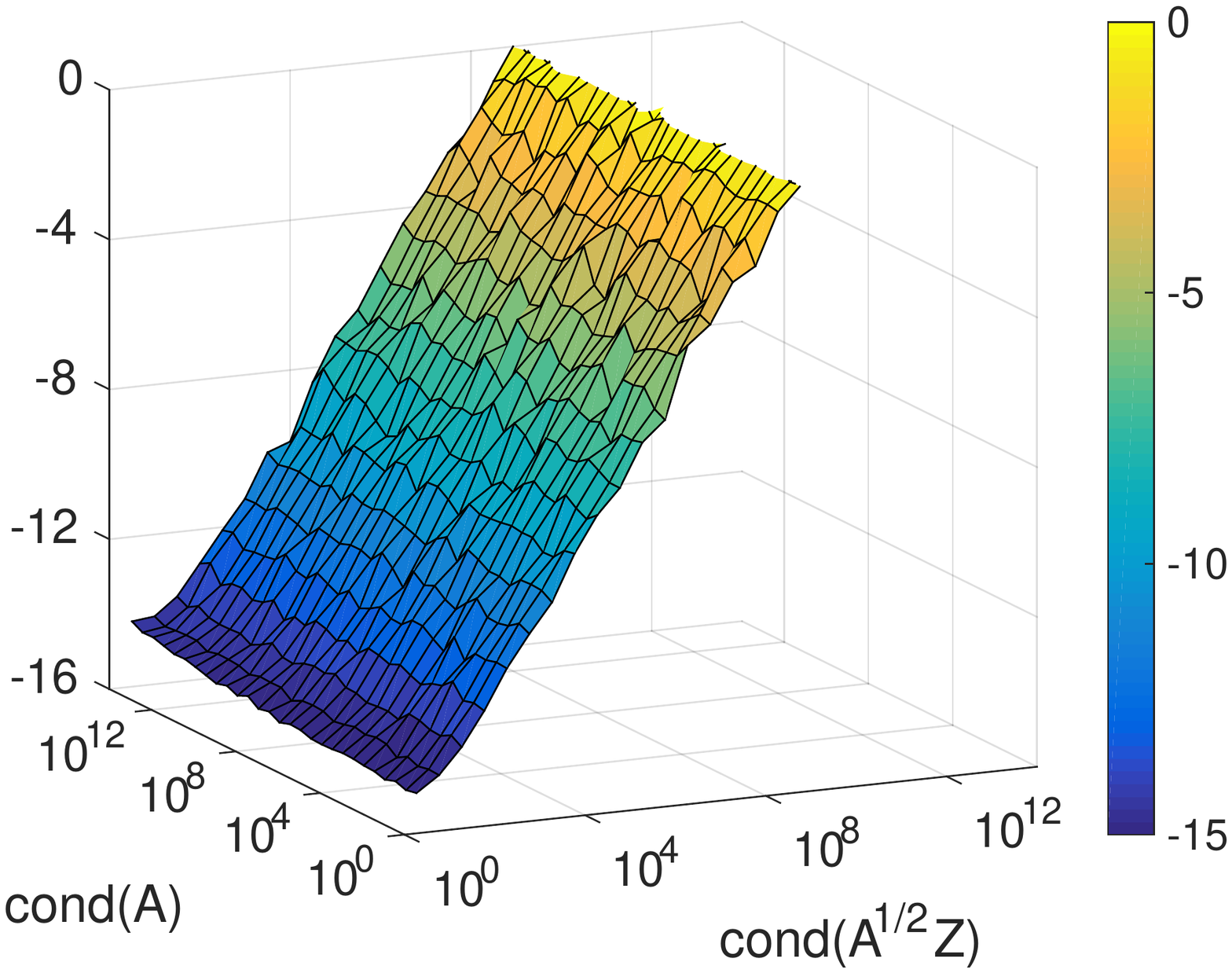}
}
\caption{
        Log10 of loss of $A$-orthogonality $\| \widehat{Q}^{\rm T} A \widehat{Q} - I_n\|$ for the case 1 that provides a best case bound.
}
\label{fig:case1}
\end{figure}
\begin{figure}[!t]
\centering
\subfloat[MGS-naive]{
\includegraphics[bb = 0 180 595 642, scale=0.30]{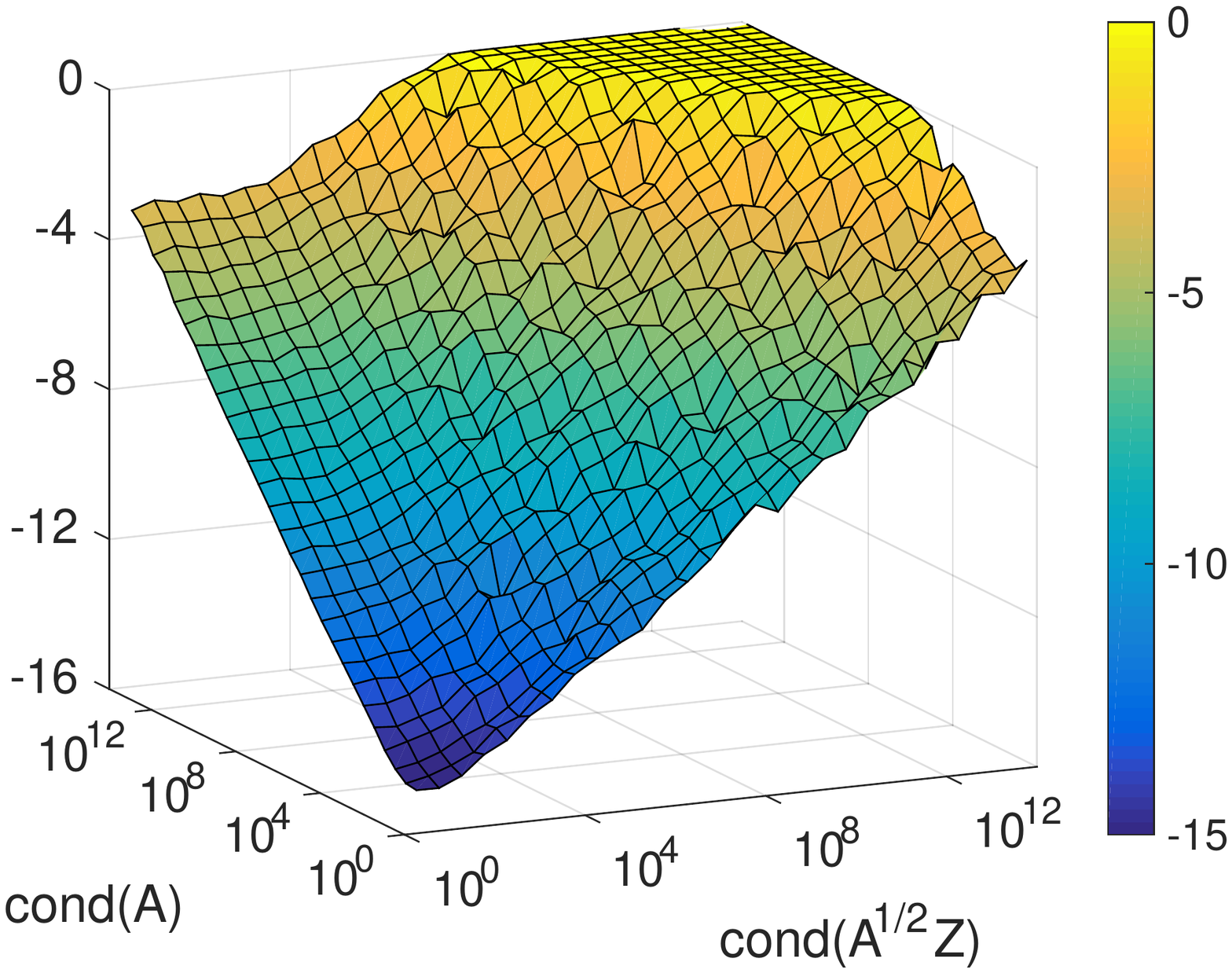}
}\\
\subfloat[MGS-HA]{
\includegraphics[bb = 0 180 595 642, scale=0.30]{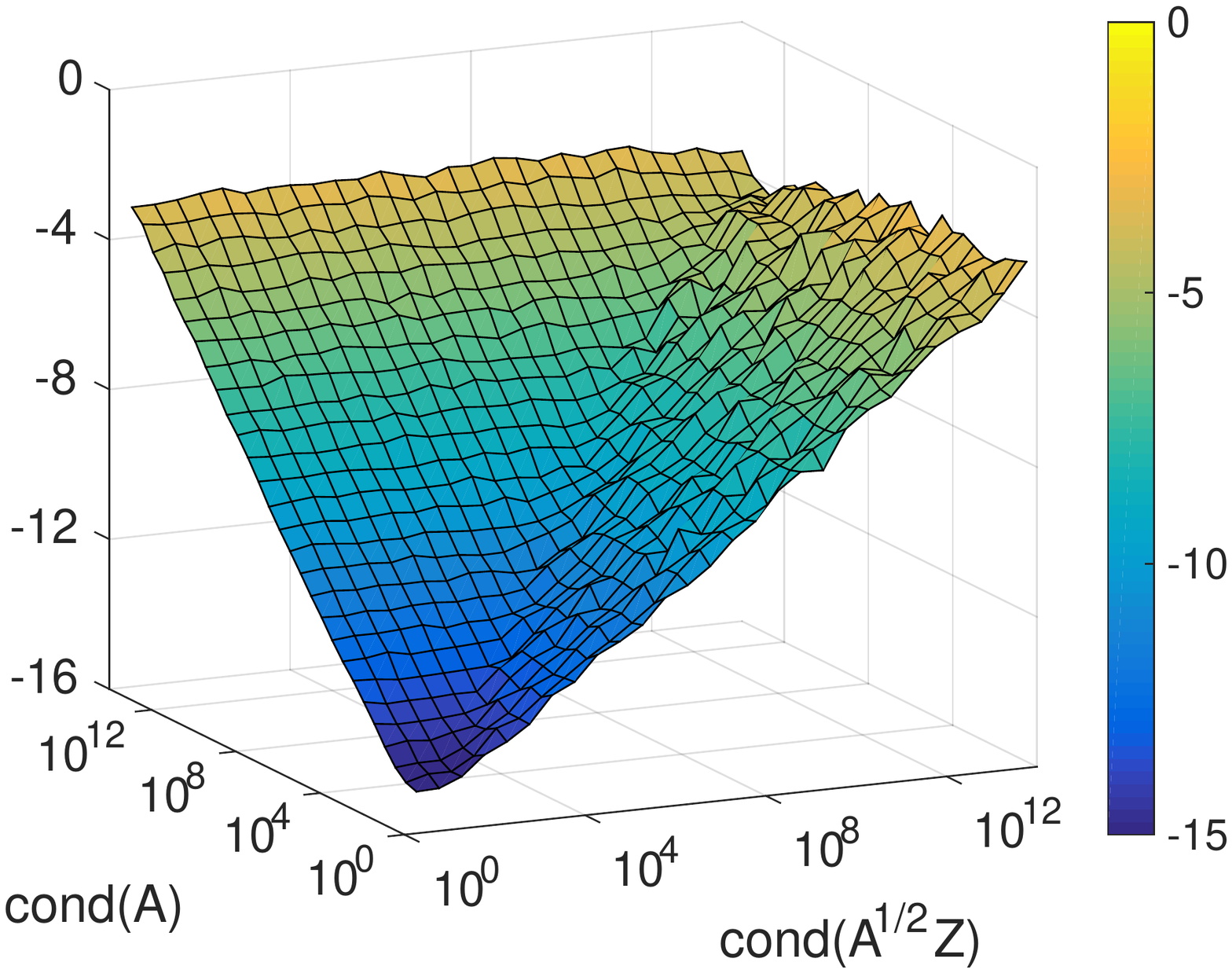}
}
\subfloat[MGS-HP]{
\includegraphics[bb = 0 180 595 642, scale=0.30]{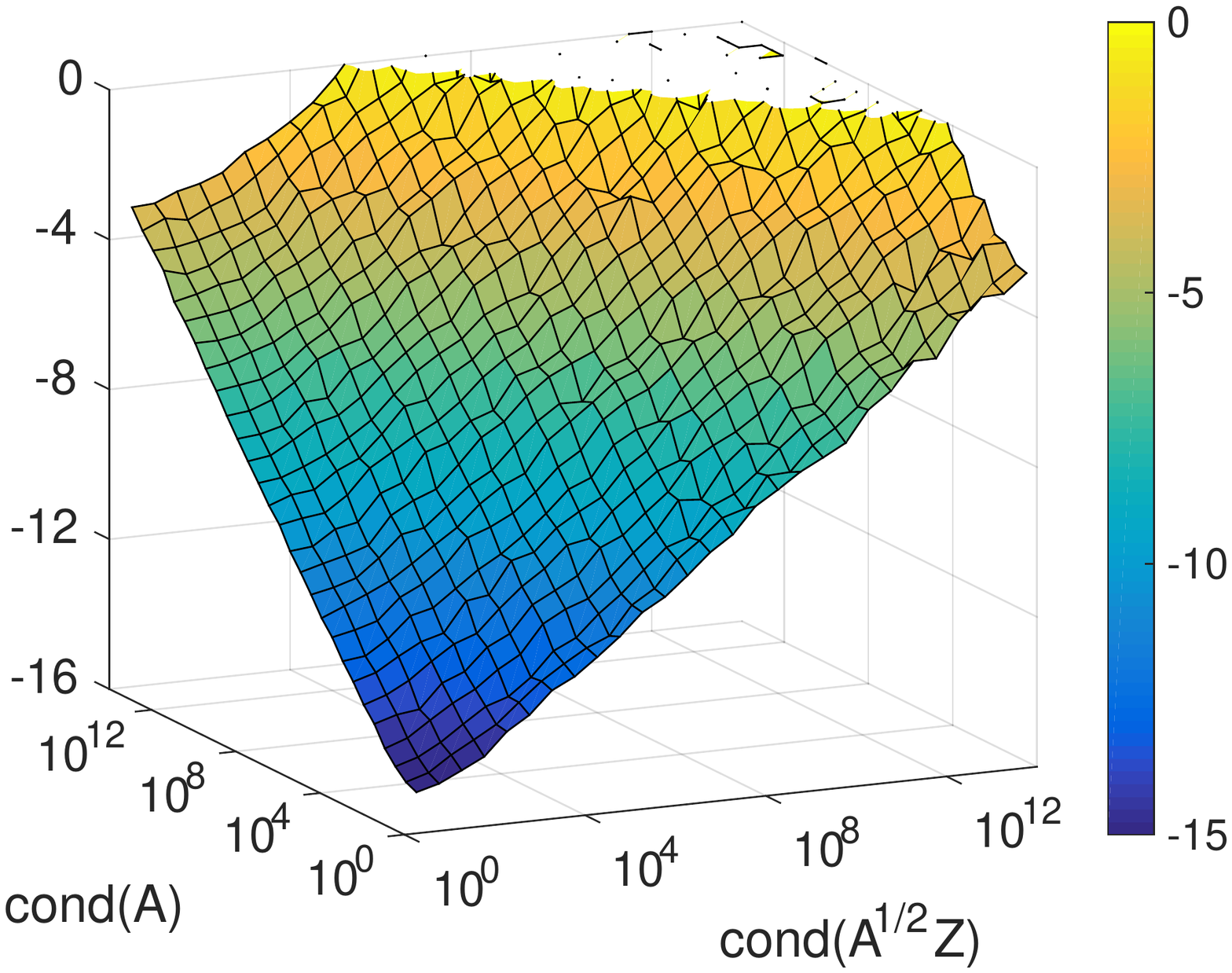}
}\\
\subfloat[Cholesky QR]{
\includegraphics[bb = 0 180 595 642, scale=0.30]{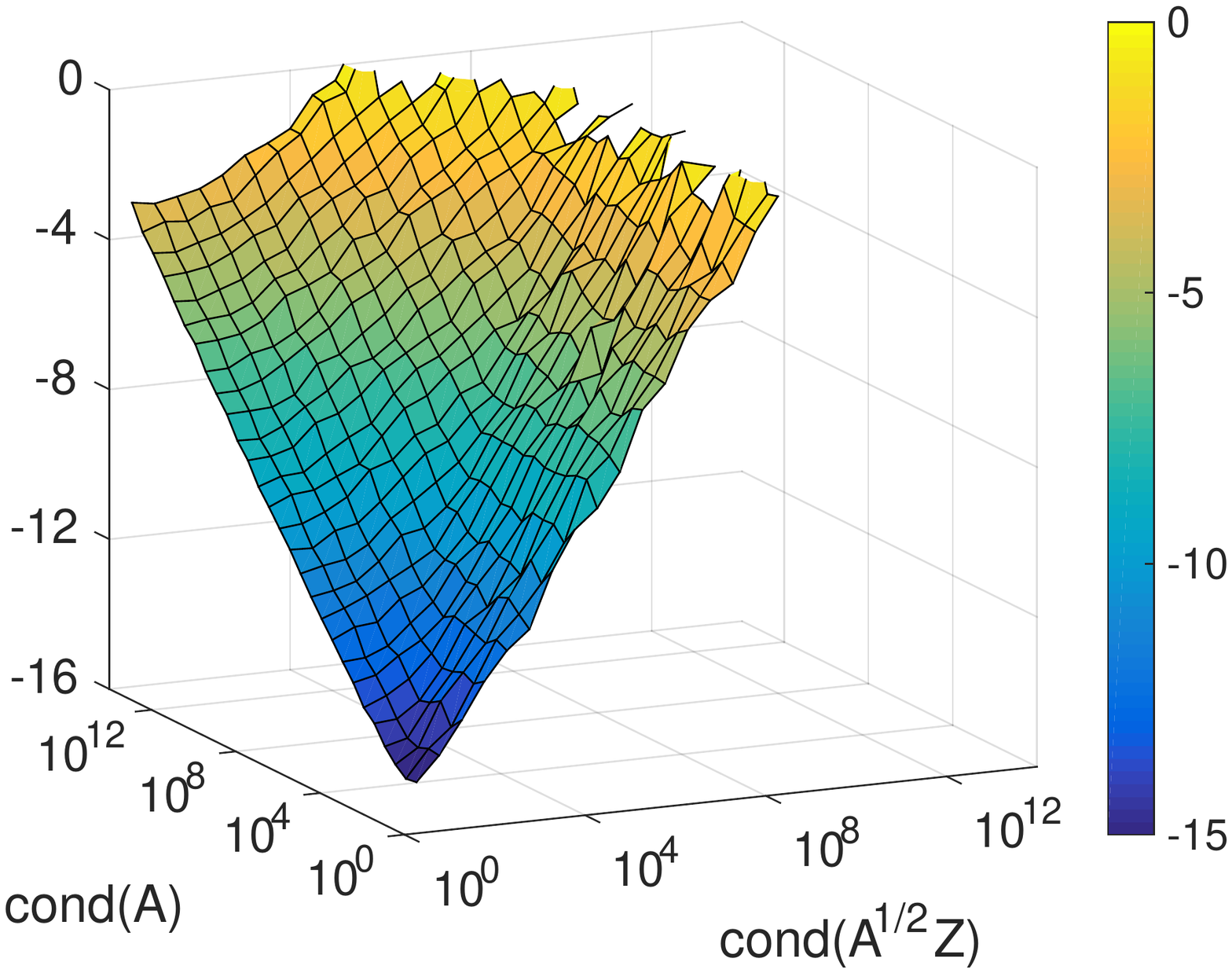}
}
\subfloat[CGS-naive]{
\includegraphics[bb = 0 180 595 642, scale=0.30]{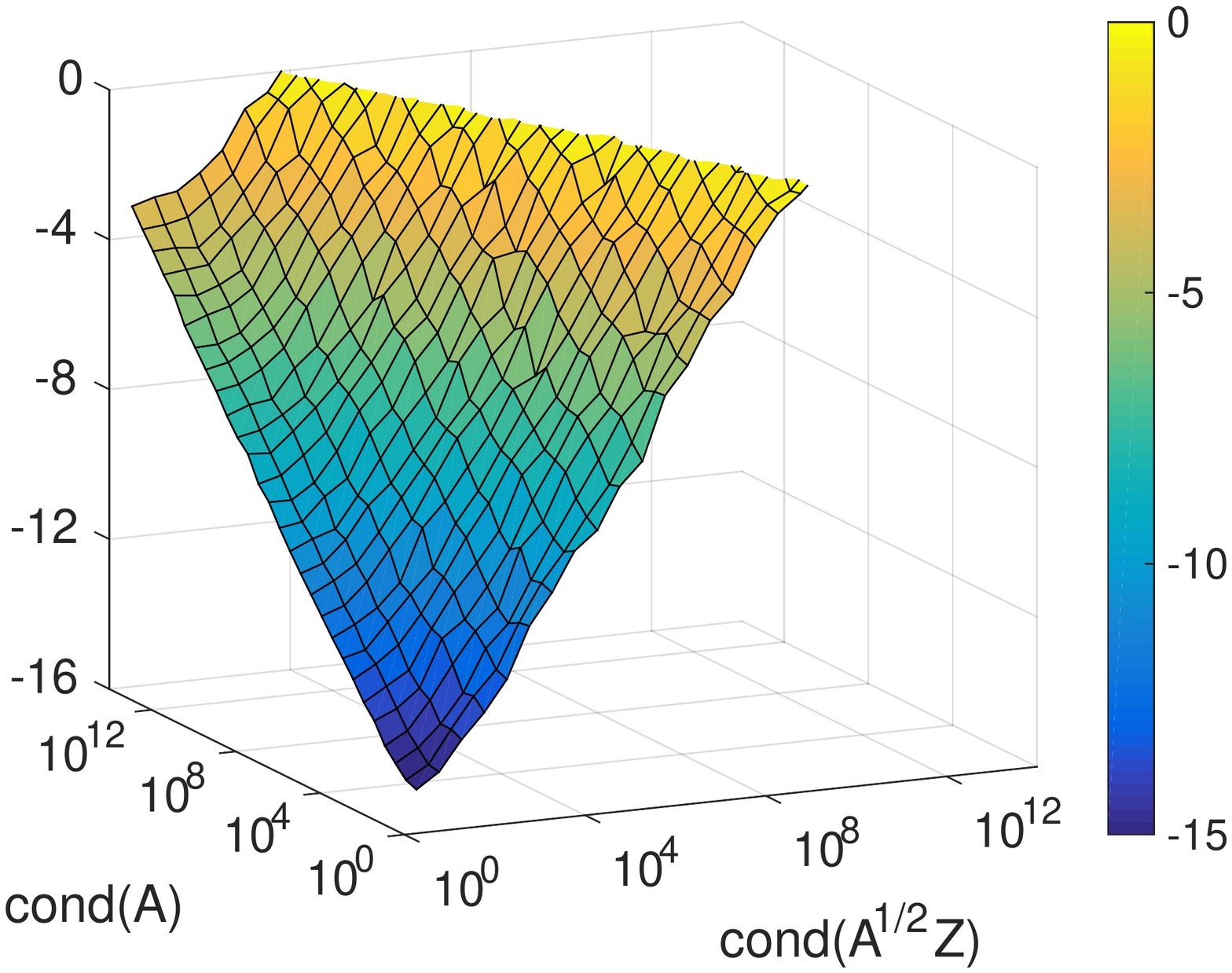}
}
\caption{
        Log10 of loss of $A$-orthogonality $\| \widehat{Q}^{\rm T} A \widehat{Q} - I_n\|$ for the case 2 that provides a worst case bound.
}
\label{fig:case2}
\end{figure}
\par
We present log10 of the loss of $A$-orthogonality as a function of $\kappa(A)$ and $\kappa(A^{1/2} Z)$ for case 1 and case 2 in Figures~\ref{fig:case1} and \ref{fig:case2}, respectively.
In case 1 (the best case), the loss of $A$-orthogonality of all methods depends only on $\kappa(A^{1/2} Z)$; in contrast, it depends on both $\kappa(A^{1/2} Z)$ and $\kappa(A)$ in case 2 (the worst case).
In both cases, MGS-naive, MGS-HA and MGS-HP show better accuracy than CGS-naive and Cholesky QR: the dependence on $\kappa(A^{1/2} Z)$ is linear for the former and quadratic for the latter.
Here, we note that Cholesky QR failed when $\kappa(A^{1/2} Z) \geq 10^8$.
\par
Next, we compare the proposed implementations, MGS-HA and MGS-HP, with MGS-naive.
MGS-HP shows nearly the same accuracy as MGS-naive in both cases and MGS-HA shows nearly the same accuracy as MGS-naive in case 1.
In addition, as a remarkable result, we observe that MGS-HA shows better accuracy than MGS-naive in case 2, especially when both $A$ and $Z$ are ill-conditioned: $\kappa(A), \kappa(A^{1/2} Z) \gg 1$; see Figure~\ref{fig:case2}(b).
%
\subsection{Numerical experiment III}
\begin{figure}[!t]
\centering
\includegraphics[bb = 0 180 595 642, scale=0.30]{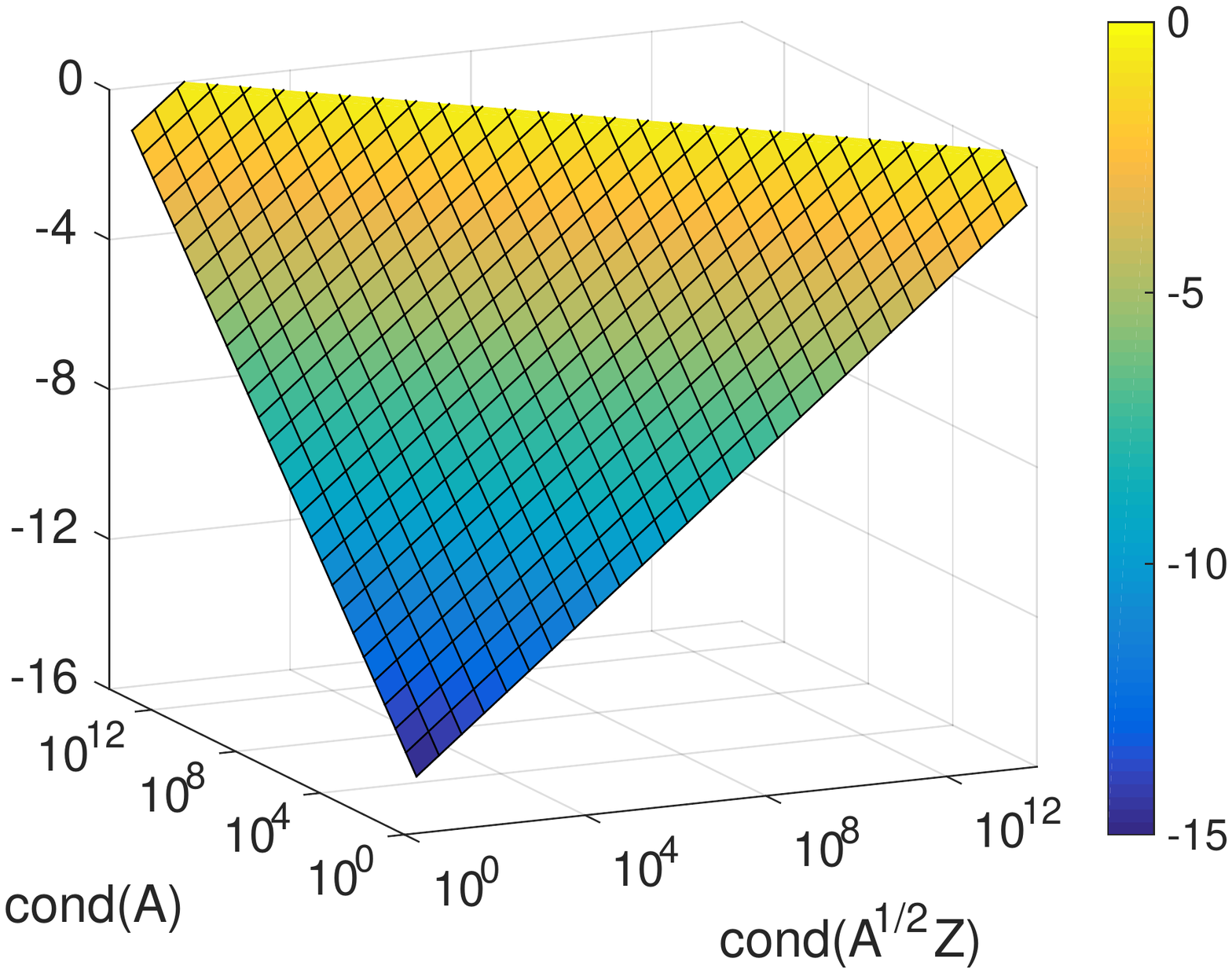}
\caption{
        Log10 of $\delta_1 := {\bf u} \kappa(A) \kappa(A^{1/2}Z)$.
}
\label{fig:thm}
%
\vspace*{12pt}
\centering
\subfloat[MGS-naive]{
\includegraphics[bb = 0 0 360 216, scale=0.5]{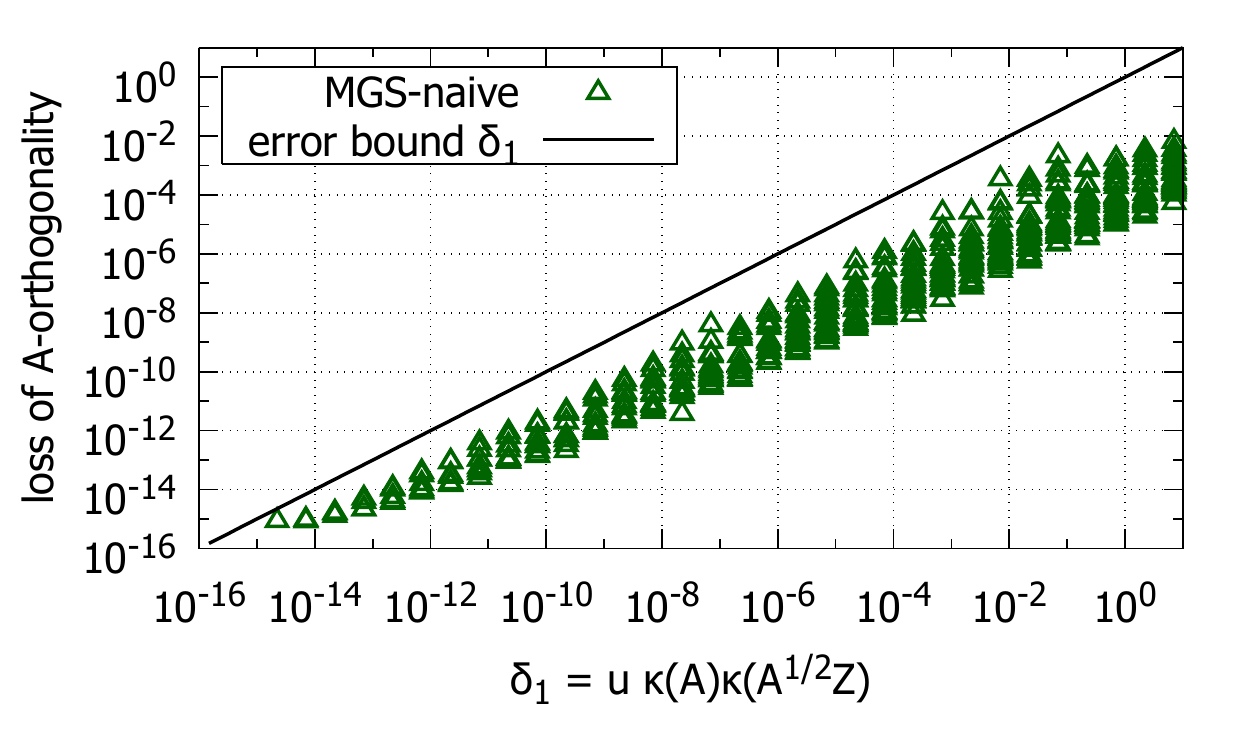}
}\\
\subfloat[MGS-HA]{
\includegraphics[bb = 0 0 360 216, scale=0.5]{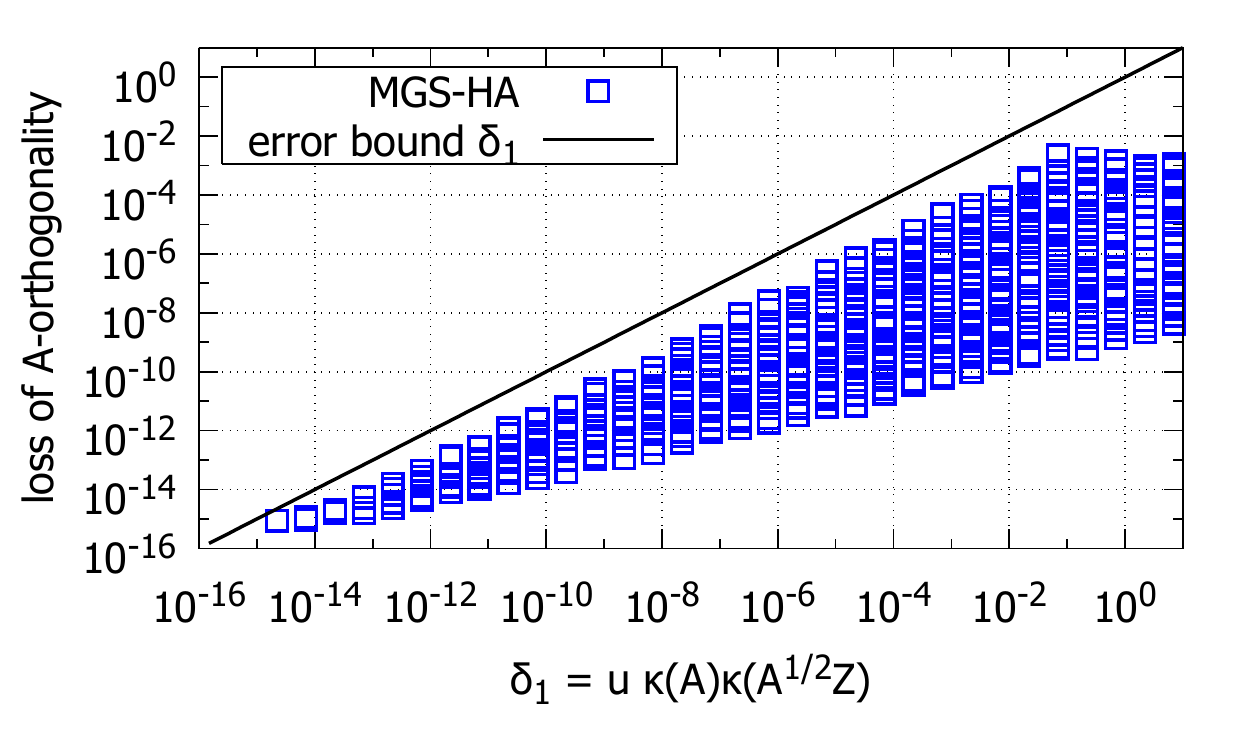}
}
\subfloat[MGS-HP]{
\includegraphics[bb = 0 0 360 216, scale=0.5]{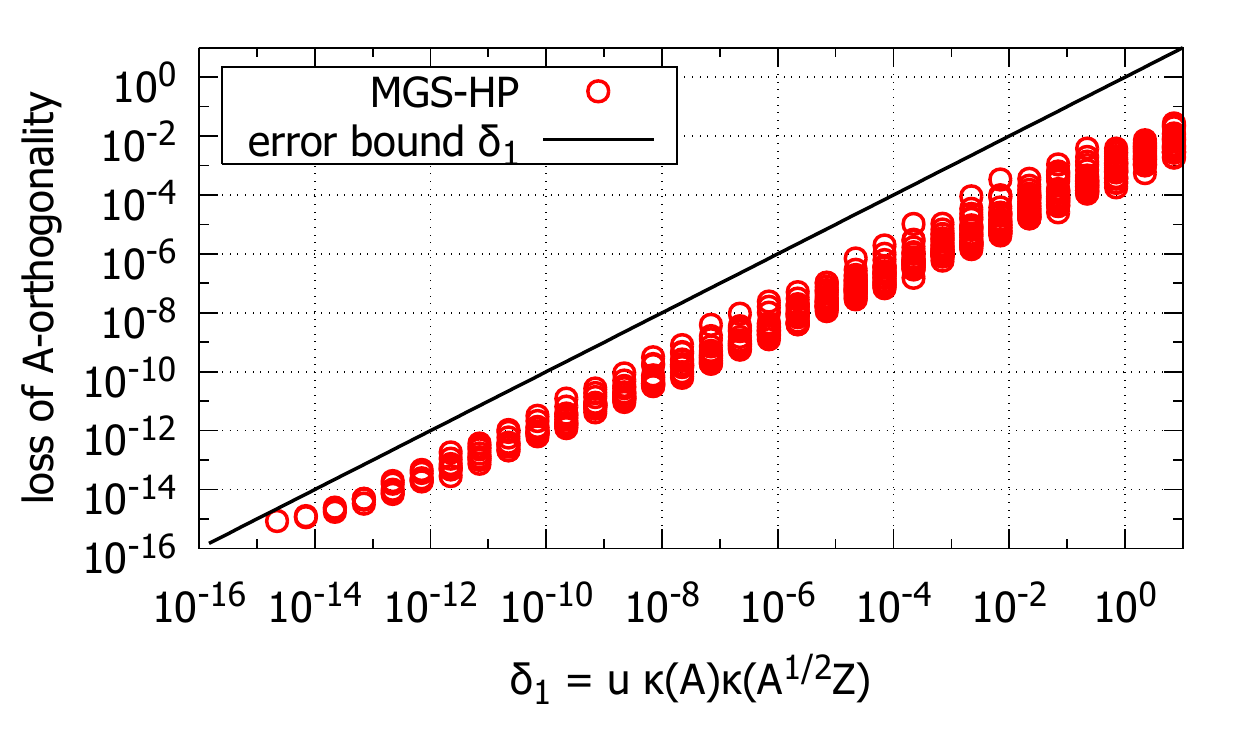}
}
\caption{
        Comparison between $\delta_1 := {\bf u} \kappa(A) \kappa(A^{1/2}Z)$ and the loss of $A$-orthogonality.
}
\label{fig:bound}
\end{figure}
Here, we compare the loss of $A$-orthogonality of the computed results of MGS-naive, MGS-HA and MGS-HP in case 2 with the theoretical error bound \eqref{eq:bound_a} derived in Section~\ref{sec:error}.
As shown in Section~\ref{sec:error}, the loss of $A$-orthogonality of MGS-naive and MGS-HA are bounded by
\begin{equation*}
        \| \widehat{Q}^{\rm T} A \widehat{Q} - I_n \| \leq \delta_1, \quad
        \delta_1 := {\bf u} \kappa(A)\kappa(A^{1/2}Z).
\end{equation*}
Figure~\ref{fig:thm} shows the upper bound $\delta_1$ as a function of $\kappa(A)$ and $\kappa(A^{1/2}Z)$, while Figure~\ref{fig:bound} plots the actual loss of $A$-orthogonality against $\delta_1$.
Comparing Figures~\ref{fig:case2}(a), (c) and \ref{fig:thm} reveals that $\delta_1$ represents the actual loss of $A$-orthogonality well for MGS-naive and MGS-HP.
In fact, we can see from Figure~\ref{fig:bound}(a), (c) that $\delta_1$ is not only an upper bound, but also a good estimate of the actual loss of $A$-orthogonality for MGS-naive and MGS-HP.
For MGS-HA, however, there are many computational results for which the loss of $A$-orthogonality is much lower than suggested by $\delta_1$; see Figure~\ref{fig:bound}(b).
This indicates that although $\delta_1$ is certainly an upper bound for MGS-HA, it may not be a sharp upper bound.
\begin{figure}[!t]
\centering
\includegraphics[bb = 0 180 595 642, scale=0.30]{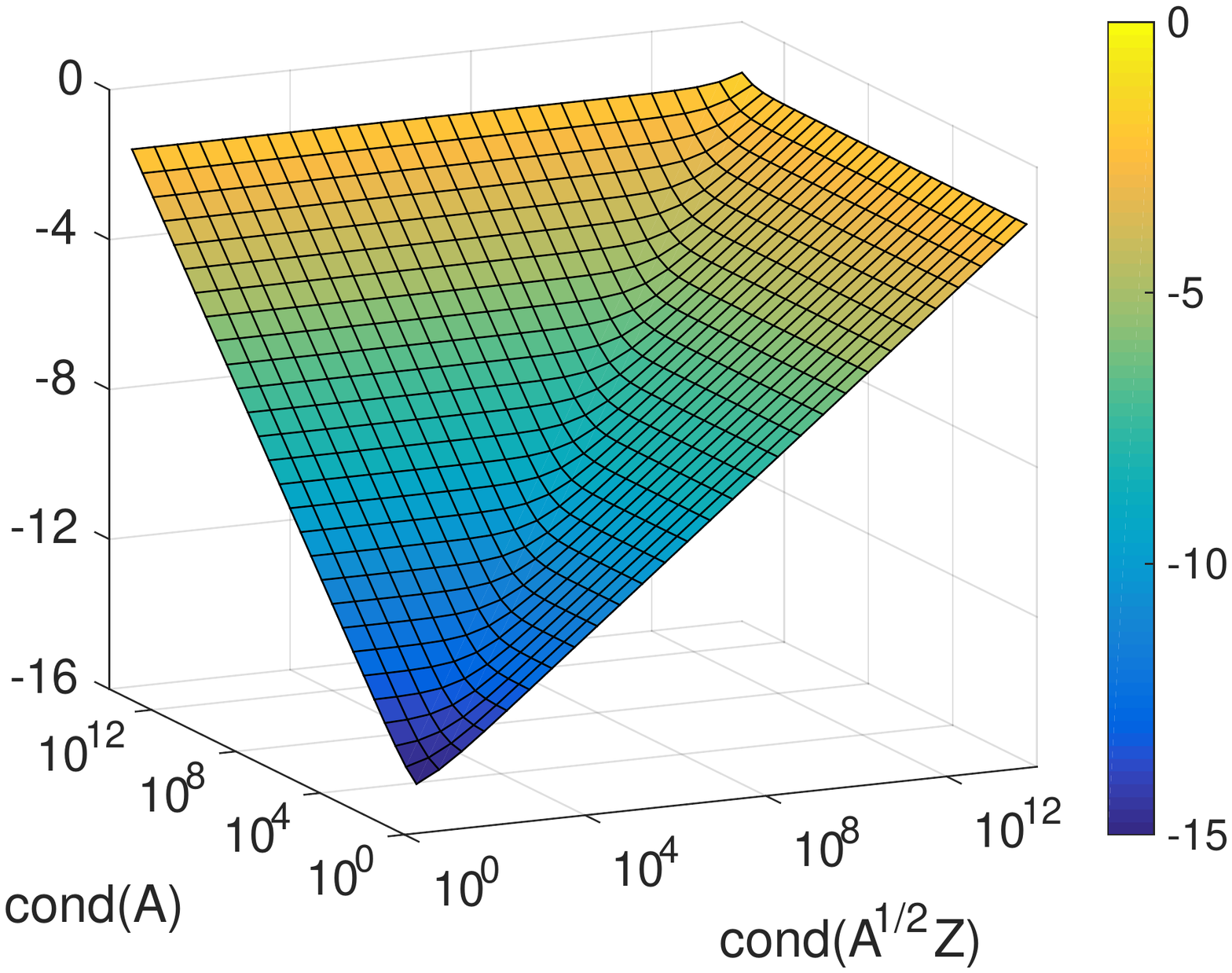}
\caption{
        Log10 of $\delta_2 := {\bf u} \left(\kappa(A) + \kappa(A^{1/2}Z) \right)$.
}
\label{fig:thm2}
%
\vspace*{12pt}
\centering
\subfloat[MGS-naive]{
\includegraphics[bb = 0 0 360 216, scale=0.5]{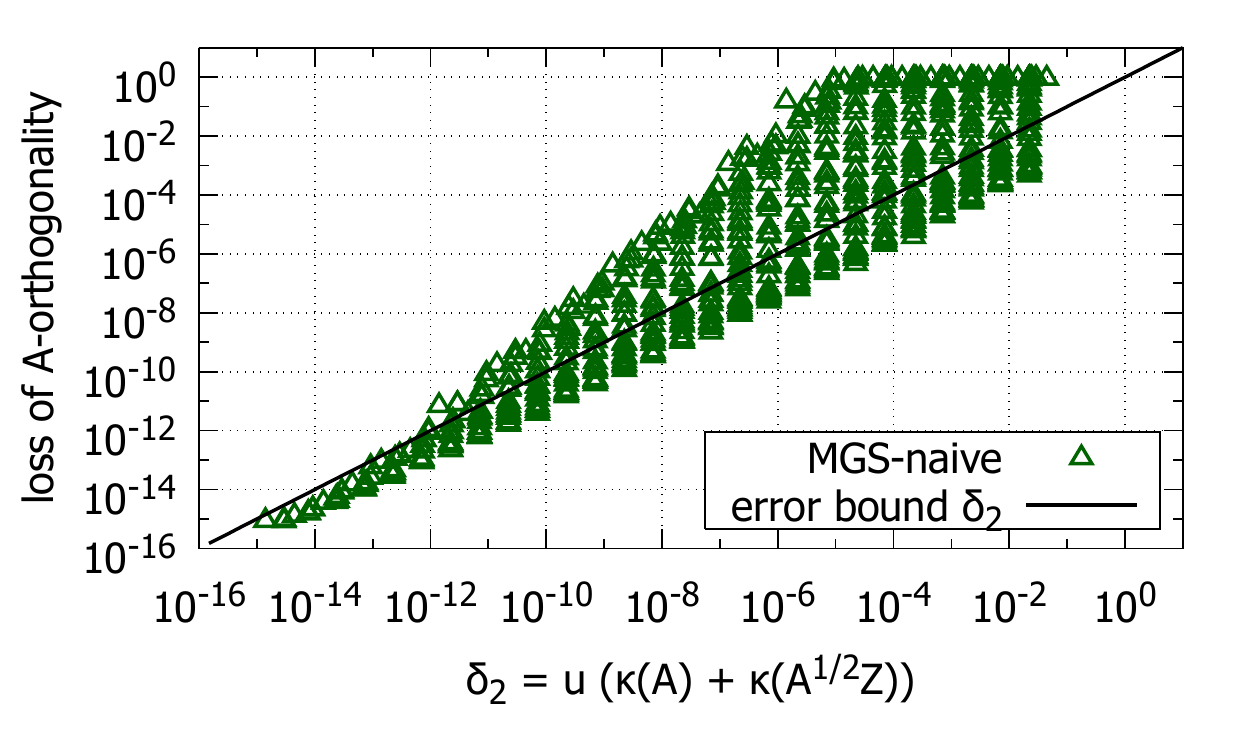}
}\\
\subfloat[MGS-HA]{
\includegraphics[bb = 0 0 360 216, scale=0.5]{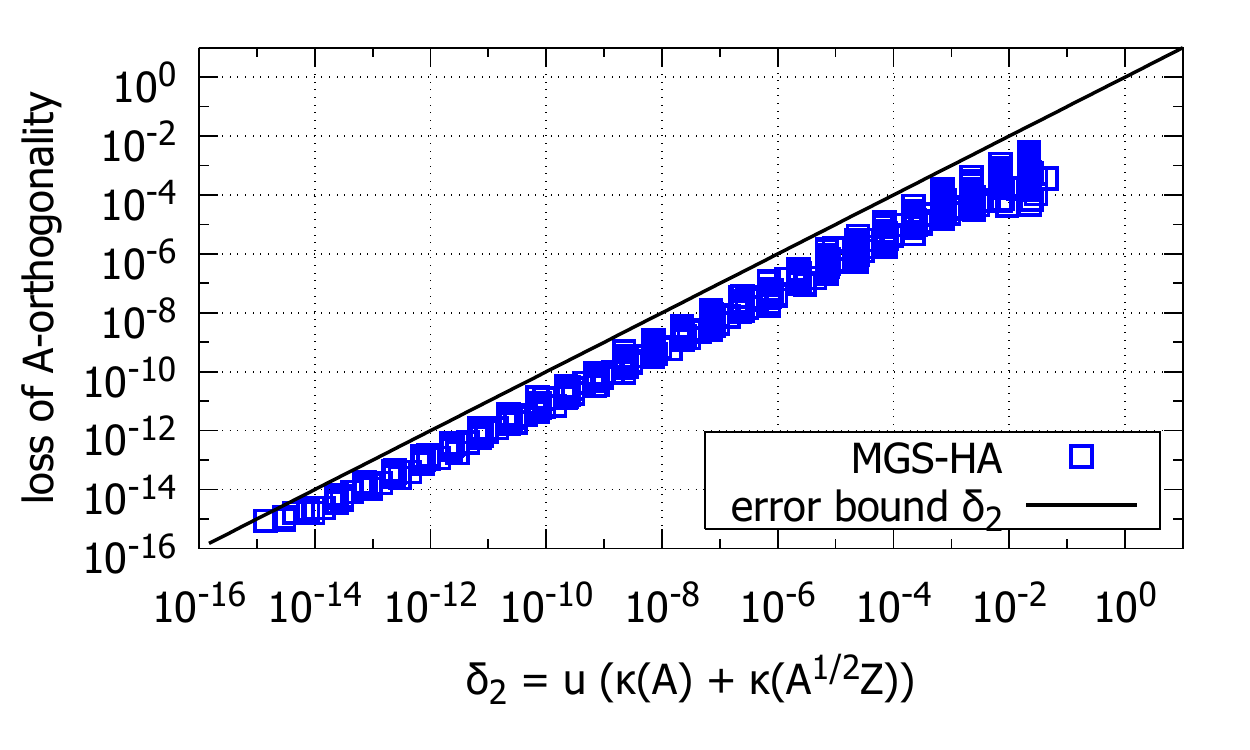}
}
\subfloat[MGS-HP]{
\includegraphics[bb = 0 0 360 216, scale=0.5]{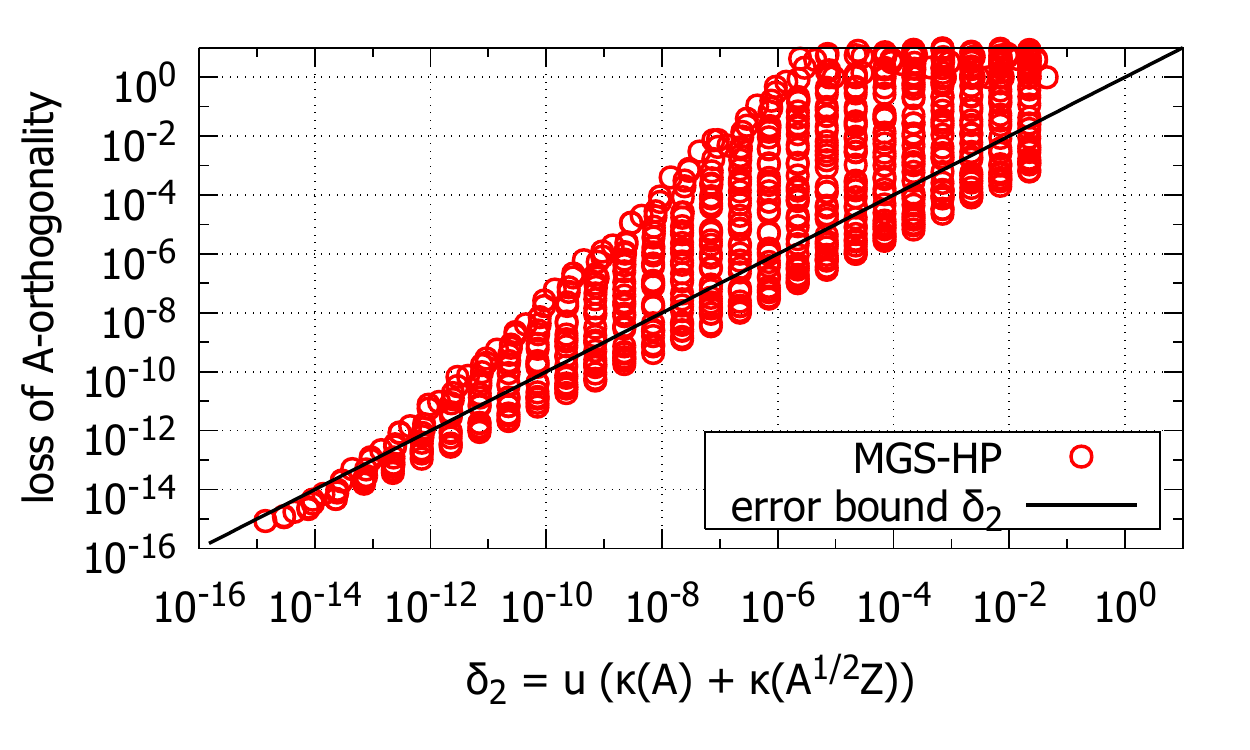}
}
\caption{
        Comparison between $\delta_2 := {\bf u} \left( \kappa(A) + \kappa(A^{1/2}Z)\right)$ and the loss of $A$-orthogonality.
}
\label{fig:bound2}
\end{figure}
\par
Instead of $\delta_1$, let us consider the following quantity:
\begin{equation*}
        \delta_2 := {\bf u} \left( \kappa(A) + \kappa(A^{1/2}Z) \right).
\end{equation*}
Figure~\ref{fig:thm2} shows $\delta_2$ as a function of $\kappa(A)$ and $\kappa(A^{1/2}Z)$, while Figure~\ref{fig:bound2} plots the actual loss of $A$-orthogonality against $\delta_2$.
Comparing Figures~\ref{fig:case2}(b) and \ref{fig:thm2}, we can see that $\delta_2$ represents the computational results of MGS-HA well.
We also observe, from Figure~\ref{fig:thm2}(b), that $\delta_2$ is a very sharp upper bound for MGS-HA.
It is to be noted that $\delta_2$ does not serve as an upper bound of the loss of $A$-orthogonality for MGS-naive and MGS-HP; see Figure~\ref{fig:bound2}(a), (c).
\par
Although the quantity $\delta_2$ we introduced here has no theoretical background yet, the numerical results suggest that it can describe the actual loss of $A$-orthogonality for MGS-HA very well. Based on this fact, we make the following conjecture on a sharper upper bound for the loss of $A$-orthogonality of MGS-HA.
\begin{conjecture}
        The loss of $A$-orthogonality of MGS-HA can be bounded as
        \begin{equation*}
        \| \widehat Q^{\rm T} A \widehat{Q} - I_n \| \leq
        \frac{\mathcal{O}(m^{3/2}) {\bf u} \left( \kappa(A) + \kappa(A^{1/2}Z) \right)}{1-\mathcal{O}(m^{3/2}) {\bf u} \left( \kappa(A) + \kappa(A^{1/2}Z)\right)} 
        \approx \mathcal{O}(m^{3/2}) {\bf u} \left( \kappa(A) + \kappa(A^{1/2}Z) \right).
        \end{equation*}
\end{conjecture}
\section{Conclusions}
\label{sec:conclusion}
In this paper, we propose two types of efficient implementations of the modified Gram-Schmidt orthogonalization with a non-standard inner product.
These methods, named MGS-HA and MGS-HP, require only $n$ MV, in contrast to the naive implementation, MGS-naive, that requires $2n$ MV.
Experimental results show that both methods are much faster than MGS-naive.
Specifically, MGS-HP is nearly as fast as Cholesky QR for small $n$.
Regarding accuracy, we prove that MGS-HA has nearly the same error bounds for representation error and loss of $A$-orthogonality as MGS-naive.
According to the numerical experiments, MGS-HP shows nearly the same accuracy and MGS-HA shows higher accuracy than MGS-naive.
We also introduce a conjecture on a sharper upper bound for the loss of $A$-orthogonality for MGS-HA (Conjecture 1).
\par
In the future, we expect to prove Conjecture 1 and also derive an upper bound for MGS-HP.
We also plan to evaluate the computational performance of the proposed implementations for large problems in parallel environments.

\section*{Acknowledgment}
The present study is supported in part by Japan Science and Technology Agency, ACT-I (No.~JPMJPR16U6) and the Japanese Ministry of Education, Culture, Sports, Science and Technology, Grant-in-Aid for Scientific Research (Nos.~26286087, 15H02708, 15H02709, 16KT0016).

\end{document}